\documentclass[12pt,a4paper,reqno]{amsart}

\usepackage{a4wide}

\usepackage{amssymb, bbold, mathrsfs}
\usepackage{amssymb}
\usepackage{graphicx}
\usepackage{float}

\usepackage{enumerate}
\usepackage{upref}
\usepackage[textures,backref,bookmarks=false]{hyperref}

\usepackage[ansinew]{inputenc}

\theoremstyle{plain}
\newtheorem{theorem}{Theorem}[section]
\newtheorem{maintheorem}{Theorem}
\newtheorem{lemma}[theorem]{Lemma}
\newtheorem{proposition}[theorem]{Proposition}

\newtheorem*{conjecture}{Question}

\theoremstyle{definition}
\newtheorem{definition}[theorem]{Definition}

\theoremstyle{remark}

\def\R{\ensuremath{\mathbb R}}
\def\N{\ensuremath{\mathbb N}}

\def\e{\ensuremath{\text e}}
\def\E{\ensuremath{\mathbb E}}
\def\m{\ensuremath{m}}

\def\B{\ensuremath{\mathcal B}}

\def\C{\ensuremath{\mathcal C}}

\def\l{\ensuremath{\text {m}}}
\def\cv{\ensuremath{\text {Cor}}}
\def\ld{\ensuremath{\text{LD}}}

\def\dist{\ensuremath{d}}

\DeclareMathOperator{\p}{P_m}
\DeclareMathOperator{\interior}{int}
\def\Esup{\mathop{\hbox{\rm essup}}}
\def\Einf{\mathop{\hbox{\rm essinf}}}
\def\a{\alpha}
\def\osc{\mathop{\hbox{\rm osc}}}

\def\p{P_m}

\def\e{\varepsilon}
\numberwithin{equation}{section}

\begin{document}

\title[]{From Rates of mixing to recurrence times \\
via large deviations}

\author{José F. Alves}
\address{José F. Alves\\ Departamento de Matematica, Faculdade de Ciências da Universidade do Porto\\
Rua do Campo Alegre 687, 4169-007 Porto, Portugal}
\email{jfalves@fc.up.pt} \urladdr{http://www.fc.up.pt/cmup/jfalves}

\author{Jorge M. Freitas}
\address{Jorge M. Freitas\\ Departamento de Matematica, Faculdade de Ciências da Universidade do Porto\\
Rua do Campo Alegre 687, 4169-007 Porto, Portugal}
\email{jmfreita@fc.up.pt}
\urladdr{http://www.fc.up.pt/pessoas/jmfreita}

\author{Stefano Luzzatto}
\address{Stefano Luzzatto\\ Mathematics Department, Imperial College\\
180 Queen's Gate, London SW7, UK}
\email{luzzatto@ictp.it}
\urladdr{http://www.ictp.it/$\sim$luzzatto}
\curraddr{Abdus Salam International Centre for Theoretical Physics, Strada Costiera 11, 34151 Trieste, Italy. }

\author{ Sandro Vaienti}
\address{Sandro Vaienti\\ UMR-6207 Centre de Physique Th\'eorique, CNRS,
Universit\'es d'Aix-Marseille I, II, Universit\'e du Sud, Toulon-Var and FRUMAM, F\'ed\'ederation de
Recherche des Unit\'es de Math\'ematiques de Marseille\\ CPT, Luminy Case 907, F-13288
Marseille Cedex 9, France}
\email {vaienti@cpt.univ-mrs.fr}

\date{\today}

\thanks{Work carried out at CIRM, ICTP, Imperial College and University of Porto. JFA and JMF were partially supported by FCT through CMUP, by POCI/MAT/61237/2004 and by PTDC/MAT/099493/2008. JMF was partially supported by FCT grant SFRH/BPD/66040/2009.}

\subjclass{37A05, 37C40, 37D25}

\keywords{Gibbs-Markov structure, decay of correlations, large deviations}

\begin{abstract}
A classic approach in dynamical systems is to use particular geometric structures to deduce
statistical properties, for example the existence of  invariant measures with stochastic-like behaviour such as large deviations or decay of correlations.
Such geometric structures are generally highly non-trivial and thus a natural question is the extent to which this approach can be applied.  In this paper  we show that in many cases stochastic-like behaviour itself implies that the system has certain non-trivial geometric properties, which are therefore necessary as well as sufficient conditions for the occurrence of the statistical properties under consideration.
As a by product of our techniques we also obtain some new results on large deviations for certain classes of systems which include Viana maps and multidimensional piecewise expanding maps.
\end{abstract}

\maketitle


\section{Introduction and statement of results}

Let  \(
f\colon M\to M\)
 be a piecewise $C^{1+}$ endomorphism defined on a
Riemannian manifold $M$, and let $\m$ denote a normalized volume
form on the Borel sets of $M$ that we call {\em Lebesgue measure.}
Here  \( C^{1+} \) denotes the class of continuously differentiable maps with H\"older continuous derivative and the precise conditions on the ``piecewise'' will be stated below.
A basic problem is the study of the statistical properties of the map \( f \), starting from questions about the
existence of an ergodic invariant measure \( \mu \) which is absolutely continuous with respect to Lebesgue to more sophisticated properties such as the rate of decay of correlations or large deviations with respect to this measure \( \mu \). In a fundamental paper \cite{Y3}, Young showed that the existence of such a measure \( \mu \) and, more significantly, the rate of decay of correlations of \( \mu \) can be deduced from the ``geometry'' of \( f \), more specifically from the existence and properties of a ``Young tower'' or ``induced Gibbs-Markov map''.
The verification of this geometric structure is of course generally  highly non-trivial, and over the last ten years a substantial number of papers have been devoted to this goal under various kinds of assumptions and using a variety of techniques \cite{Y2,Y3, BLS, ALP, Gou, Hol, DHL}. Combining these geometric constructions with the abstract results of Young, and more recent results concerning also other statistical properties such as large deviations \cite{RY, MN}, much more significant progress has been made in understanding the stochastic-like behaviour of deterministic dynamical systems in the last ten years than had been since the pioneering results on uniformly hyperbolic systems in the 60's and early~70's.

A natural question concerns the limitations of this approach. Might there be large classes of systems, or even specific ``pathological'' systems,
that exhibit certain statistical properties but for which this approach does not and cannot work because such systems just do not admit the required geometrical structures?
The main purpose of this paper is to show that in many cases such systems do not exist, and that in fact stochastic-like behaviour such as decay of correlations at certain rates is in itself sufficient to imply  the existence of an induced Gibbs-Markov map with the corresponding properties.
This geometry is therefore both necessary and sufficient for the statistical properties of the system.   We will now give the precise formulation of these results.

\subsection{Main definitions}
We start with the definition of a Gibbs-Markov structure and then give the formal definitions of the notion of decay of correlations and large deviations.

\begin{definition}\label{def:inducing-scheme}
We say that $f$ admits a
 \emph{Gibbs-Markov induced map}  if there exists a ball  \( \Delta \subset M \), a countable partition \( \mathcal P \) (mod 0) of \(
   \Delta \) into topological balls \( U \) with smooth boundaries,
   and a return time function \( R: \Delta \to \mathbb N \)
   constant on elements of  \(\mathcal P \) satisfying the following
   properties:
   \begin{enumerate}
   \item {\em {Markov}:} for each \( U\in\mathcal P \) and
   \( R=R(U)
   \),
   \(
   f^{R}: U \to \Delta
   \)
   is a \( C^{1+} \) diffeomorphism (and in particular a bijection).  Thus the
   induced map
   \( F: \Delta \to \Delta\)   given by  \( F(x) = f^{R(x)}(x) \) is
   defined almost everywhere and satisfies the classical Markov property.

   \item {\emph{Uniform expansion}:} there exists \( \lambda < 1 \) such that
   for almost all \( x\in \Delta \) we have
   \( \|DF(x)^{-1}\|\leq \lambda. \)
   In particular the \emph{separation time} \( s(x,y) \) given by the maximum
   integer such that \( F^{i}(x) \) and \( F^{i}(y) \) belong to the same
   element of the partition \( \mathcal P \) for all \( i\leq s(x,y) \),
   is defined and finite for almost all  \( x,y\in\Delta \).

   \item {\em {Bounded distortion}:}
   there exists \( K >0 \) such that for any points
   \( x,y\in \Delta \) with \( s(x,y) <\infty\) we have
   \[
  \left|\frac{\det DF(x)}{\det DF(y)}-1\right| \leq
  K \lambda^{-s(F(x),
   F(y))}.
   \]
\end{enumerate}
We define that ``tail'' of the return time function at time \( n \) as the set
 \[
\mathscr R_{n}=\{x\in\Delta:\; R(x)>n\}
\]
of points whose return time is larger than \( n \), and we say that the return time function is integrable if
  \[ \int R \;dm < \infty  .\]
\end{definition}

\begin{definition}[Expanding measure]
 We say that \( \mu \) is (regularly) \emph{expanding} if
\[
\log \|Df^{-1}\|\in L^{1} \quad \text{ and } \quad \int \log\|Df^{-1}\|d\mu < 0.
 \]
 \end{definition}

A first example of the way in which geometric structure is related to statistical properties is given by the relation between the above two definitions. Indeed, it is shown in \cite{AlvDiaLuz} that for large classes of maps including multidimensional maps with ``non-degenerate'' critical points the two structures are completely equivalent in the sense that \( f \) admits a Gibbs-Markov induced map if and only if it admits a regularly expanding absolutely continuous invariant probability measure.
In this paper we develop these general philosophy further by considering more refined statistical properties.

\begin{definition}[Decay of correlations]
Let \( \mathcal B_{1}, \mathcal B_{2} \) denote Banach spaces of real valued measurable functions defined on \( M \).
We denote the \emph{correlation} of non-zero functions $\varphi\in \mathcal B_{1}$ and  \( \psi\in \mathcal B_{2} \) with respect to a measure $\mu$ as
\[
\cv_\mu (\varphi,\psi):=\frac{1}{\|\varphi\|_{\mathcal B_{1}}\|\psi\|_{\mathcal B_{2}}}\left|\int \varphi\, \psi\, d\mu-\int  \varphi\, d\mu\int
\psi\, d\mu\right|.
\]
We say that we have \emph{decay
of correlations}, with respect to the measure $\mu$, for observables in $\mathcal B_1$ \emph{against}
observables in $\mathcal B_2$ if, for every $\varphi\in\mathcal B_1$ and every
$\psi\in\mathcal B_2$ we have
 $$\cv_\mu(\varphi,\psi\circ f^n)\to 0,\quad\text{ as $n\to\infty$.}$$
 \end{definition}

We will use the notation \( \lesssim  \) to mean  \( \leq  \)
up to multiplication by a constant depending only on the map \( f \).
We say that the  decay of correlations is \emph{exponential}, \emph{stretched exponential}, or \emph{polynomial} if
it is \( \lesssim \e^{-\tau n} \),
\( \lesssim  \e^{-\tau n^{\theta}} \) or
\( \lesssim  n^{-\beta} \) respectively,  for constants \( \tau, \theta, \beta \) which depend only on \( f \). Most of the time we shall choose \( \mathcal B_{2}=L^{p} \) for \( p=1 \) or \( p = \infty \), and \( \mathcal B_{1}=\mathcal H_{\alpha} \) the space of H\"older continuous functions with H\"older constant \( \alpha \).
Recall that  the  H\"older norm of an observable \( \varphi\in\mathcal H_\alpha \) is given by
\[
 \|\varphi\|_{\mathcal H_{\alpha}} := \|\varphi\|_{\infty}+
\sup_{x\neq y}\frac{|\varphi(x)-\varphi(y)|}{|x-y|^{\alpha}}.
 \]


\subsection{Local diffeomorphisms}

We start by stating our results in the setting of \( C^{1+} \) local diffeomorphisms.

\begin{maintheorem}\label{th:corr=>tower}
Let \( f: M\to M \) be a \(
C^{1+} \)  local diffeomorphism.
Suppose that
\( f \) admits an ergodic expanding acip $\mu$;
\begin{enumerate}
\item
if
there exists \( \beta > 1\) such that
\( \cv_{\mu} (\varphi, \psi \circ f^{n}) \lesssim n^{-\beta}
\)   for every \( \varphi \in \mathcal H_{\alpha} \) and \( \psi \in L^{\infty}(\mu) \),
then there is a Gibbs-Markov induced map  with
     \(
     \m(\mathscr R_{n}) \lesssim n^{-\beta+1}.
     \)
     \end{enumerate}
Suppose  moreover
that \( d\mu/d\l \) is uniformly bounded away from 0 on its support. Then
 \begin{enumerate}
 \item[(2)]
if there exist \( \tau, \theta > 0 \)
such that  \(
\cv_{\mu} (\varphi, \psi \circ f^{n}) \lesssim e^{-\tau n^{\theta}}\)
for every \( \varphi \in \mathcal H_{\alpha} \) and \( \psi \in L^{\infty}(\mu) \), then
there is a Gibbs-Markov  induced map   with \(
  \m(\mathscr R_{n}) \lesssim e^{-\tau'n^{\theta'}}\)
for some \( \tau'>0 \) and
     $\theta'={\theta}/{(\theta+2)}.$
\end{enumerate}
\end{maintheorem}

We emphasize that these are essentially direct converses of the remarkable results
of Young \cite{Y3} where she showed that
the rate of  decay  of the tail of the return time function implies a corresponding rate for the decay of  correlations. Thus we conclude that

\begin{quote}
\emph{the rate of decay of correlations is polynomial (resp. stretched exponential) if and only if there exists a Gibbs-Markov induced map with polynomial (resp. stretched exponential) tail.}
\end{quote}

Both in our results here and in \cite{Y3} something is lost in the actual value of the constants that appear in the exponents when passing from the assumptions to the conclusion, and thus we are not able to give a complete if and only if statement including the specific rates of decay. This is probably unavoidable as the ``actual'' rates which are intrinsic to the system probably depend on finer characteristics which can be controlled to some extent by changing the exponents but not avoided completely.
We remark also that the additional assumption on the density of \( \mu \) for the (stretched) exponential case is due to the use of different technique for constructing the induced map, as we shall explain in more detail below. It holds in various known examples such  as when the map is ``locally eventually onto'', i.e.  every open set of positive \( \mu \) measure covers the support of \( \mu \) in a finite number of iterates.

\subsection{Maps with critical/singular sets}\label{subsec:critical-points}

The results stated above are for local diffeomorphisms, and are already relevant and non-trivial in that setting,  but there exist many interesting examples which may fail to be local diffeomorphisms due to the presence of critical points
(where \( \det Df = 0 \)), singular points (where $Df$ does not exist or \( \|Df\|=\infty \)) or discontinuities of \( f \). We shall generally denote the collection of all such points as the critical/singular set. Most of the results which deduce statistical information from Gibbs-Markov maps apply equally to systems with a non-empty critical/singular set; in fact this is one of the strengths of this approach, the partition structure of Gibbs-Markov induced maps allows in some sense to avoid bad regions of the phase space. For the converse results, the situation is in principle more complicated because we need to show that a Gibbs-Markov map can still be constructed and that possible accumulation of images or preimages of the critical/singular set do not adversely affect the decay rates of tail of the return times.  We shall show that in fact most of
the results stated above do essentially apply under some mild assumption on the critical/singular set and on the density of the measure \( \mu \).

\begin{definition}
We say that \( x\) is a \emph{critical point} if  \( Df(x) \) is not invertible
and a \emph{singular point} if   \( Df(x) \)  does not exist.
We let \( \mathcal C \) denote the set of critical/singular points and
let  \(
\dist(x, \mathcal C) \) denote the distance between the point \( x\in M \)
and the set \( \mathcal C \).
We say that a set \( \mathcal C \) of critical/singular points is
\emph{non-degenerate} if
there are constants \( B, d>0 \) such that for all \( \epsilon > 0 \)
\begin{enumerate}
\item[(C0)]\quad
 \( \displaystyle
 \m\left(\{x: \dist(x, \mathcal C) \leq \epsilon\}\right) \leq B \epsilon^{d }
 \) \quad (in particular  \( \m(\mathcal C) =0 \));
\end{enumerate}
and there exists  $\eta>0$
such that for every $x\in
M\setminus\mathcal C $ and  \( v\in T_{x}M \) with \( \|v\|=1 \) we have
\begin{enumerate}
\item[(C1)]
\quad $ \displaystyle B^{-1}\dist(x,\mathcal C)^{\eta}\leq
\|Df(x)v\| \leq B\dist(x,\mathcal C )^{-\eta}$.
\end{enumerate}
Moreover, for all \( x,y \in M\setminus \mathcal C \) with \( \dist(x, \mathcal C)\leq \dist(y, \mathcal C) \) we have
\begin{enumerate}
\item[(C2)] \quad $\displaystyle{\left|\log\|Df(x)^{-1}\|-
\log\|Df(y)^{-1}\|\:\right|\leq B\left|\log(\dist(y,\C))-\log(\dist(x,\C))\right|}$;
\item[(C3)]
\quad $\displaystyle{\left|\log|\det Df(x)|- \log|\det
Df(y)|\:\right|\leq B\left|\log(\dist(y,\C))-\log(\dist(x,\C))\right|}$.
\end{enumerate}
\end{definition}
We remark that the conditions (C2) and (C3) imply the corresponding conditions used \cite{ABV, ALP, Gou}.
As long as the critical set satisfies the above mild non-degeneracy assumptions, we recover essentially the results stated above for local diffeomorphisms in the polynomial and stretched exponential case.

\begin{maintheorem}\label{th:CS-corr=>tower}
Let \( f: M\to M \) be a \(
C^{1+} \)  local diffeomorphism
outside
a nondegenerate critical set \( \mathscr C \).
Suppose that
\( f \) admits an ergodic expanding acip \( \mu \)
with \( d\mu/d\l \in L^{p}(m) \) for some \( p>1 \);
\begin{enumerate}
\item
if
there exists \( \beta > 1\) such that
\( \cv_{\mu} (\varphi, \psi \circ f^{n}) \lesssim n^{-\beta}
\)   for every \( \varphi \in \mathcal H_\alpha \) and \( \psi \in L^{\infty}(\mu) \),
then for any $\gamma>0$ there is a Gibbs-Markov  induced map such that
$\m(\mathscr R_{n}) \lesssim n^{-\beta+1+\gamma}$.
\end{enumerate}
Suppose moreover that  \( d\mu/d\l \)  is uniformly bounded away from 0 on its support;
\begin{enumerate}
\item[(2)]
if there exist \( \tau, \theta > 0 \)
such that  \(
\cv_{\mu} (\varphi, \psi \circ f^{n}) \lesssim e^{-\tau n^{\theta}}\)
for every \( \varphi \in \mathcal H_\alpha \) and \( \psi \in L^{\infty}(\mu) \), then
for any $\gamma>0$ there is a Gibbs-Markov  induced map   such that
     $\m(\mathscr R_{n}) \lesssim e^{-\tau'n^{\theta'-\gamma}}$ for
     $\theta'={\theta}/{(3\theta+6)}.$
\end{enumerate}
\end{maintheorem}

Thus also in the very general setting of maps with critical and singular points we  obtain a converse to Young's results and conclude that
the rate of decay of correlations is polynomial (resp. stretched exponential) if and only if there exists a Gibbs-Markov induced map with polynomial (resp. stretched exponential) tail.

\subsection{Large deviations}

A key step in our argument is to show that the rate of decay of correlations implies certain large deviation estimates. This is itself a result of independent interest partly also because it is a completely abstract result and we use no additional structure on $M$ or $f$ other than  $f:M\to M$ being measurable and \emph{nonsingular} (see Section~\ref{subsec:PF-Koop-CE}) with respect to an ergodic probability measure $\mu$ on $M$. In particular, we need no Riemannian structure on~$M$.

\begin{definition}[Large deviations]
 Given an ergodic probability measure \( \mu \) and
 \( \epsilon>0 \)
we define the \emph{large deviation} at time~$n$ of the time
average of the observable $\varphi$ from the spatial average
~as
\[
\ld_{\mu}(\varphi,\epsilon, n):=\mu\left(\left|\frac1 n \sum_{i=0}^{n-1}
\varphi\circ f^n-\int\varphi d\mu \right|>\epsilon\right).
\]
\end{definition}
By Birkhoff's ergodic theorem  the quantity \( \ld_{\mu}(\varphi,\epsilon,n) \to 0 \), as
\( n\to \infty \), and a relevant question also in this case  is the rate of this decay.

\begin{maintheorem}\label{th:DC=>LD}

Let \( f: M\to M  \) preserve an ergodic probability measure \( \mu \)  with respect to which $f$ is nonsingular. Let  \( \B\subset L^{\infty}(\mu) \) be a Banach space with norm \( \|\cdot \|_{\B} \) and $\varphi\in \B$.
\begin{enumerate}
\item
 Let $\beta>0$ and suppose that  for all $\psi\in L^\infty(\mu)$ we have
 \(
 \cv_\mu(\varphi,\psi\circ f^n) \lesssim n^{-\beta}.
 \)
 Then, for every
 \( \epsilon>0 \), there exists  \( C=C(\varphi, \epsilon)>0 \)
 such that \(
 \ld_{\mu}(\varphi,
\epsilon,n) \leq  C n^{-\beta}.
 \)
\item
 Let $\theta, \tau>0$ and suppose that for all $\psi\in L^\infty(\mu)$ we have
 \(
 \cv_\mu(\varphi,\psi\circ f^n)\lesssim e^{-\tau n^\theta}.
 \)
Then, for every \( \epsilon>0 \)
there exist  \( C=C({\varphi, \epsilon})>0\) and
\( \tau'=\tau'({\tau,\varphi, \epsilon})>0 \)
such that
 \(
 \ld_{\mu}(\varphi,\epsilon,n)
 \leq C  e^{-\tau' n^{{\theta}/{(\theta+2)}}}.
 \)
 \end{enumerate}
\end{maintheorem}

In the course of the proof of this theorem we shall obtain explicit formulas for the constants which appear in the large deviation bounds. These formulas will play an important role in the application of the results to the the proof of the other theorems.
We remark that a version of the polynomial case has been proved in \cite[Theorem 1.2 and Lemma 2.1]{M}, however due to our need for a very explicit form of the constants   we include a fully worked out proof here.

We also give below a straightforward application of this result to obtain an estimate for the large deviation for the well known class of Viana maps.
There have been several recent results concerning large deviations for nonuniformly expanding maps, see \cite{AraPac, MN, RY}, but remarkably none of them actually apply to this specific class of maps.

\subsubsection{Viana maps}
An important class of nonuniform expanding dynamical
    systems (with critical sets) in dimension greater than one was
    introduced by Viana in \cite{V}. This has served as a model
  for some relevant results on the ergodic
    properties of
    non-uniformly expanding maps in higher dimensions; see
    \cite{Alv00,AA,ABV,AV}.
 This class of  maps
 can be described as
follows. Let $a_0\in(1,2)$ be such that the critical point $x=0$
is pre-periodic for the quadratic map $Q(x)=a_0-x^2$. Let
$S^1=\mathbb R/\mathbb Z$ and $b:S^1\rightarrow \mathbb R$ be a Morse function, for
instance, $b(s)=\sin(2\pi s)$. For fixed small $\alpha>0$,
consider the map
 \[ \begin{array}{rccc} \hat f: & S^1\times\mathbb R
&\longrightarrow & S^1\times \mathbb R\\
 & (s, x) &\longmapsto & \big(\hat g(s),\hat q(s,x)\big)
\end{array}
 \]
 where  $\hat q(s,x)=a(s)-x^2$ with
$a(s)=a_0+\alpha b(s)$, and $\hat g$ is the uniformly expanding
map of the circle defined by $\hat{g}(s)=ds$ (mod $\mathbb Z$) for some
large integer $d$. In fact, $d$ was chosen greater or equal to 16
in \cite{V}, but recent results in \cite{BST} showed that some
estimates in \cite{V} can be  improved and $d=2$ is enough. It is
easy to check that for $\alpha>0$ small enough there is an
interval $I\subset (-2,2)$ for which $\hat f(S^1\times I)$ is
contained in the interior of $S^1\times I$. Thus, any map $f$
sufficiently close to $\hat f$ in the $C^0$ topology has
$S^1\times I$ as a forward invariant region. We consider from here
on these maps restricted to $S^1\times I$ and we call any such map a \emph{Viana map}.
It was shown in \cite{Alv00,AV} that Viana maps have a unique ergodic expanding acip \( \mu \).

\begin{theorem}
Let \( f \) be a Viana map and let \( \mu \) be its unique expanding acip. Then, for every \( \epsilon>0 \) there exists \( \tau, C>0 \) such that for all \( \varphi \in \mathcal H_{\alpha} \) we have
\[ LD_{\mu}(\varphi, \epsilon, n) \leq C e^{-\tau n^{1/5}}.
 \]
\end{theorem}

As observed for example in \cite{ALP}, Viana maps satisfy the non-degeneracy conditions on the critical set. Moreover,
it is proved in \cite{Gou} that every Viana map exhibits stretched exponential decay of correlations, with \( \theta=1/2  \), for H\"older continuous functions against \( L^{\infty}(\mu) \) functions.
The theorem is then  a direct application of part (2) of Theorem \ref{th:DC=>LD}.

\subsection{Exponential estimates}

The results given above do not yield exponential estimates and it is not clear  at the moment if this is just a technical issue or there is some deeper reason. However it turns out that we can get exponential estimates if we assume that the correlation decay is uniformly summable against all \( L^{1} \) observables.

\begin{maintheorem}\label{th:exp}
Let \( f: M\to M  \) preserve an ergodic probability measure \( \mu \)  with respect to which $f$ is nonsingular. Let  \( \B\subset L^{\infty}(\mu) \) be a Banach space with norm \( \|\cdot \|_{\B} \) and $\varphi\in \B$.
Suppose that there exists \( \xi(n) \) with $\sum_{n=0}^\infty\xi(n)<\infty$ such that for all $\psi\in L^1(\mu)$ we have
 \(
  \cv_\mu(\varphi,\psi\circ f^n) \leq \xi(n).
  \) Then
\begin{enumerate}
\item
there exists
 \(\tau=\tau(\varphi)>0 \) and,  for every \( \epsilon>0 \),
  there exists \( C=C({\varphi, \epsilon})> 0 \) such that
 \(
 \ld_{\mu}(\varphi,\epsilon,n) \leq C e^{- \tau n}.
 \)
 \end{enumerate}
 Suppose moreover that \( f \) is a \( C^{1+} \) local diffeomoprhism, \( d\mu/d\l \) is uniformly bounded away from 0 on its support, and \( \B=\mathcal H_{\alpha} \) is the space of H\"older continuous maps.  Then
 \begin{enumerate}
 \item[(2)]
there exists a Gibbs-Markov  induced map   with
\(
  \m(\mathscr R_{n}) \lesssim e^{-\tau'n}\)
for some \( \tau'>0 \).
\end{enumerate}
\end{maintheorem}


There are some fairly general classes of piecewise expanding maps which exhibit summable (in fact exponential) decay of correlations against
\( L^1 \) functions, and to which therefore these results apply. We give some explicit examples in Appendix \ref{ap.pecs}. Here we state some  general conditions in terms of the properties of the Perron-Frobenius operator. We will show that all the examples of Appendix \ref{ap.pecs} satisfy these conditions and in particular that Gibbs-Markov maps satisfy these conditions.  However the following question is still an open problem.

\begin{conjecture}
Suppose there is a Gibbs-Markov induced map with \( \l(\mathscr R_{n}) \lesssim e^{-\tau'n}\)
for some \( \tau'>0 \). Is there \( \xi(n) \) with $\sum_{n=0}^\infty\xi(n)<\infty$ such that
 \(
  \cv_\mu(\varphi,\psi\circ f^n) \leq \xi(n)
  \)
 for every \( \varphi \in  \mathcal H_{\alpha} \) and \( \psi \in L^{1} (\mu)\)?
\end{conjecture}

If the question above has an affirmative answer, then the results given above would yield essentially an equivalence also in the exponential case between
exponential decay of correlations and having an induced Gibbs-Markov map with exponential tail.

\subsubsection{Perron-Frobenius}\label{se.perron} Let  $M$ be a measurable space (at this stage $M$ needs not to be a Riemannian manifold)  endowed with a reference probability measure  $m$  on a $\sigma$-algebra  $\mathcal{M}$, and let $f: M\to M$ be a measurable map. Consider the usual Perron-Frobenius operator
  $\p: L^1(m)\rightarrow  L^1(m)$ as  in Appendix~\ref{ap:A}.
 Assume that there is a seminorm $|\cdot|_{\mathcal B}$  on $L^1(m)$ such that:
\begin{enumerate}
\item $\mathcal B=\{\varphi\in L^1(m): |\varphi|_{\mathcal B}<\infty\}$ is a Banach space with the norm
$$
\|\cdot\|_{\mathcal B}= |\cdot|_{\mathcal B}+ \|\cdot\|_{L^1(m)};
$$
\item $\mathcal B$ is {adapted} to $L^1(m)$:  the inclusion $\mathcal B \hookrightarrow  L^1(m)$ is compact;
\item $\p(\mathcal B)\subset \mathcal B$ and $\p|_\mathcal B$ is bounded with respect to the norm $\|\cdot\|_{\mathcal B}$;
\item  {Lasota-Yorke} inequality holds: there are $n_0\ge 1$, $0<\alpha<1$ and $\beta>0$ such that
$$
|\p^{n_0}\varphi|_{\mathcal B}\le \alpha |\varphi|_{\mathcal B}+\beta \|\varphi\|_{L^1(m)}, \quad\forall \varphi\in \mathcal B;
$$
\item $\mathcal B$ is a  Banach  algebra with the norm $\|\cdot\|_{\mathcal B}$; in particular, there is $C>0$ such that
$$\|\varphi\psi \|_{\mathcal B}\le C \|\varphi\|_{\mathcal B} \|\psi\|_{\mathcal B},\quad\forall\,\varphi,\psi\in\mathcal B;$$
\item $\mathcal B$ is continuously injected in $L^{\infty}(m)$:
there exist a constant $C'>0$ such that
$$\|\varphi\|_{L^{\infty}(m)}\le C'\|\varphi\|_B,\quad\forall \varphi\in \mathcal B.$$
\end{enumerate}

\begin{theorem}\label{th:piecewise}
 Let $f:M\to M$ verify conditions (1)-(6).
 Then \(f\) exhibits exponential decay of correlations against observables in \( L^1(\mu)\). Assume moreover that \( d\mu/d\l \)  is uniformly bounded away from 0.  Then, in particular, for every \( \epsilon>0 \) there exists \( \tau, C>0 \) such that for all \( \varphi \in \mathcal B \) we have
\[ LD_{\mu}(\varphi, \epsilon, n) \leq C e^{-\tau n}.
 \]
\end{theorem}

The proof the first part of this Theorem is relatively standard and we include it in the Appendix in Section~\ref{sec:piecewise}. The second part then follows by a direct application of Theorem~\ref{th:exp}.

\subsubsection{Intermittent maps}
Finally we give an application of our results to show that one-dimensional intermittent maps \emph{cannot} exhibit summable decay of correlations agains \( L^1\) functions.
Let \( f: S^{1} \to S^{1}  \) be a \( C^{1+} \) local diffeomorphism of the circle satisfying \( f'(x)>1 \) for all \( x\neq 0 \) and such that
\[ f(x) \approx x+|x|^{1+\gamma}
 \]
in some neighbourhood of \( 0 \),  for some \( \gamma\in (0,1) \). We remark that the notation  \( \approx \) is used here to indicate the fact that \( f \) in a neighbourhood of 0 is equal to \( x+|x|^{1+\gamma} \) plus higher order terms and the first and second derivative of the higher order terms are still of higher order.

This is a very well known and well studied class of maps, first introduced in \cite{PomMan}. They are well known to have a unique expanding acip \( \mu \).
Their decay of correlations has been studied in detail and been shown to be at least polynomial for several classes of observables in several papers, we mention for example
\cite{LivSauVai} for \( C^{1} \) observables, in \cite{Y3} for H\"older continuous observables.

 \begin{theorem}\label{th:sum}
Suppose there exists \( \xi(n) \) such that
\[ \cv_{\mu}(\varphi, \psi\circ f^{n}) \leq \xi(n) \]
  for all \( \varphi\in \mathcal H_{\alpha} \) and \( \psi\in L^{1}(\mu)  \).
Then
\[ \sum_{n=1}^{\infty} \xi(n) = \infty.
 \]
 \end{theorem}

This follows by contradiction from  Theorem \ref{th:exp}. Indeed, this states that summable decay of correlations  against all \( L^{1} \) functions implies the existence of a Gibbs-Markov induced map with exponential tail of the return times. By \cite{Y3} this implies exponential decay of correlations for all H\"older continuous observables. However,  it is proved in \cite{Hu}, see also \cite{Sar}, that the decay of correlations cannot be faster than polynomial: there exist Lipschitz functions \( \varphi, \psi: S^{1}\to \mathbb R \) such that \( \cv_{\mu}(\varphi, \psi\circ f^{n}) \geq C n^{1-1/\gamma} \). This gives rise to a contradiction and thus Theorem \ref{th:sum} holds.

\subsection{Strategy and overview}
In  Section \ref{sec:decdev}
we prove Theorem \ref{th:DC=>LD} and part (1) of Theorem~\ref{th:exp}, namely the fact that decay of correlations imply large deviations. These are abstract results of an essentially probabilistic nature and can be formulated in terms of bounds on sums of random variables. In particular we shall apply here a result of Azuma and Hoeffding (see Appendix~\ref{ap:A}) on large deviations for a sequence of martingale differences. To apply these arguments in the exponential case we
need to use that \( (P^n_\mu\varphi )_n\) is summable in \( L^\infty(\mu) \) for every \( \varphi\in L^\infty(\mu) \), where \( P_\mu \) is the Perron-Frobenius operator, and we can show this under the assumption of summable decay of correlation against \(L^1(\mu)\) functions as stated in Theorem~\ref{th:exp}.

The application of Theorem \ref{th:DC=>LD} to the proof of Theorem
\ref{th:corr=>tower} and to the second part of Theorem~\ref{th:exp} is formulated in Theorem \ref{th:LD=>tower} and proved in Section \ref{sec:locdif}. This is relatively straightforward since in the case of \(C^{1+} \) local diffeomorphisms, the function \( \log \|Df^{-1}\| \) is H\"older continuous and therefore, from Theorem~\ref{th:DC=>LD} satisfies large deviations either with a polynomial or stretched exponential rate or, from the first part of Theorem~\ref{th:exp}, with an exponential rate. We show that such large deviation rates for
\( \log \|Df^{-1}\| \) imply the assumptions of the constructions of Gibbs-Markov induced maps in \cite{ALP, Gou} which therefore yield the desired result.

The situation in the presence of critical points or singularities is significantly more complicated. We still eventually show that the assumptions of \cite{ALP, Gou} are satisfied, but in this case we need large deviation estimates for both functions \( \log \|Df^{-1}\| \) and
\( -\log \dist(x, \mathcal C) \), where \( \dist(x, \mathcal C) \) denotes the distance to the critical/singular set, neither of which in this case are H\"older continuous. In Theorem~\ref{th:LD=>tower2} in Section~\ref{sec:unbdev} we assume for the moment large deviation estimates (polynomial, stretched exponential, and exponential) for these two functions and show how to obtain the construction of the Gibbs-Markov maps with the required tail estimates, and thus in particular
deduce the proof of Theorems~\ref{th:CS-corr=>tower}  in this setting.

In Proposition \ref{prop:DC=>LD-nonempty-CS} which we prove in Section ~\ref{sec:nonhold}, we use an approximation argument to obtain large deviation estimates for the two particular functions we are interested in, even though they are not H\"older continuous,  using the fact that we have the estimates for H\"older continuous functions. Technically, it is exactly at this point that we lose the exponential estimates and are thus not able to prove an exponential version of the second part of Theorem~\ref{th:exp} for systems with critical or singular points.

In Appendix~\ref{ap:A}, we give standard definitions and notation concerning Perron-Frobenius operators and martingales, and state the two main probabilistic theorems which we apply in the paper. In Appendix~\ref{ap.pecs} we give several classes of piecewise expanding maps which satisfy the assumptions of Theorem~\ref{th:piecewise} above.

We conclude this introduction with some brief remarks concerning the assumption that \( d\mu/dm \) is bounded away from zero on its support, which appears in the Theorems \ref{th:corr=>tower}, \ref{th:CS-corr=>tower} and~\ref{th:exp}, specifically when dealing with stretched exponential and exponential estimates. This is due to some quite subtle differences between the construction of induced Markov maps in \cite{ALP}
where polynomial estimates are obtained, and \cite{Gou}, where stretched exponential and exponential (as well as polynomial) estimates are obtained. Both papers work with essentially the same set of assumptions but the construction of \cite{Gou} is in some sense more ``global", thus requiring an assumption on the density \( d\mu/dm \) on all of its support. On the other hand, it is possible to prove that the density
\( d\mu/dm \) is  necessarily bounded away from zero in some small ball, and this is sufficient for the construction of \cite{ALP}, which is more ``local". It is not therefore clear at this point whether this assumption is merely technical.

\section{Decay of correlations imply large deviations}
\label{sec:decdev}

In this section we prove Theorem \ref{th:DC=>LD}. Assume that  $f:M\to M$ is measurable and nonsingular with respect to an ergodic acip  $\mu$ defined on a $\sigma$-algebra $\mathcal M$ of $M$, and let \( \B\subset L^{\infty}(\mu) \) be a Banach space. Let
\( \varphi\in \mathcal B \)  and suppose
 without loss of generality that \( \int \varphi d\mu = 0 \).
For \( n\in\mathbb N \) we write
\begin{equation}
\label{eq:def-Sn}
S_{n}= \sum_{i=0}^{n-1} \varphi \circ f^{i} .
\end{equation}
We are therefore
interested in an upper bound for \( \mu(|S_{n}|>\epsilon n) \).
The idea of the proof of Theorem~\ref{th:DC=>LD} is to write $S_n$  as the sum of martingale differences plus some error terms that can be controlled by means of the assumption on the rate of decay of correlations. Then, everything boils down to bound the sum of martingale differences using two abstract results known as the Rio and Azuma-Hoeffding inequalities. For the statement of these inequalities  see Theorems~\ref{th:Rio} and \ref{th:Azuma-Hoeffding}, as well as other standard notions that we will use in this section which are collected for convenience in Appendix~\ref{ap:A}. In particular, we shall use repeatedly properties (P1)-(P5) about Perron-Frobenius and Koopman operators
\[  P_\mu: L^{1}(\mu)\to L^{1}(\mu) \quad \text{ and } \quad U_\mu: L^{\infty}(\mu) \to L^{\infty}(\mu) .
\]
For notational simplicity we shall omit the measure $\mu$ in the notation for these operators and spaces. Also,
we denote by $\|\cdot\|_p$ the usual norm in $L^p(\mu)$ for $1\le p\le\infty$.
%
We define for \(  j=1,\ldots, n \)
\begin{equation}
\label{eq:def-filtration-special}
\mathcal F_j= f^{-(n-j)}\mathcal M.
\end{equation}
Observe that the measurability of $f$ does indeed imply that $\mathcal F_1\subset\mathcal F_2\subset\ldots\subset\mathcal F_n$.
Then  let
\[
X_j:= \varphi \circ f^{n-j}
\]
 Notice that the measurability of $f$ implies  that each \( X_{j} \) is measurable with respect to \( \mathcal F_{j} \) and therefore
 $\{\mathcal F_j\}_{j=1}^n$  indeed forms a filtration as defined in Appendix~\ref{ap:A}.
For every $k\in\N$  let
\begin{equation}\label{def:chi-k}
\chi^{(k)} :=\sum_{j=1}^k P^j\varphi
\quad\text{ and } \quad
\xi^{(k)}:=\varphi+\chi^{(k)}-\chi^{(k)}\circ f-P^k\varphi,
\end{equation}
and, for every \( j=1,...,n \),
\begin{equation}\label{def:Z-k-j}
Z^{(k)}_j:=\xi^{(k)}\circ f^{n-j}.
\end{equation}
It is then a tedious but straightforward exercise to check that
\begin{equation}\label{Sn}
X_j = Z^{(k)}_j+
(\chi^{(k)}\circ f^{n-j+1}-\chi^{(k)}\circ f^{n-j})+ (P^{k}\varphi)\circ f^{n-j},
\end{equation}
and therefore
\begin{equation}\label{Sn1}
S_n=\sum_{j=1}^n X_j=\sum_{j=1}^n Z^{(k)}_j+\chi^{(k)}\circ f^n-\chi^{(k)}+\sum_{j=1}^nP^k\varphi\circ f^{n-j}.
\end{equation}
We emphasize that this equality holds for every \( k \). At the moment \( k \) is a free parameter, but we shall eventually choose
\( k \) as a function of \( n \) in order to get the final estimates.
The terms above will be used in the polynomial and stretched exponential case. For the exponential case we use a similar decomposition essentially taking \( k=\infty \).
Then we write
\[
\chi :=\sum_{i=1}^\infty P^i\varphi
\quad\text{ and } \quad
\xi:=\varphi+\chi-\chi\circ f,
\]
and, for every \( j=1,...,n \),
\begin{equation}\label{eq:Zj}
Z_j:=\xi\circ f^{n-j}
\end{equation}
We remark that we will show in the exponential case that \( \chi \) is well defined and in fact lies in \( L^{\infty} \).
It is  straightforward to check that
\begin{equation}\label{Snexp}
S_n=\sum_{j=1}^n Z_j+\chi\circ f^n-\chi.
\end{equation}

 \begin{lemma}\label{lem:mart}
 \( \{Z^{(k)}_{j}\}_{j=1}^n \) is a sequence of martingale differences.
 \end{lemma}

 \begin{proof}
Clearly, $Z_j^{(k)}$ is measurable with respect to $\mathcal F_j$, for all $j=1,\ldots, n$. By property~(P1)  and the invariance of $\mu$ we have
\begin{align*}
\E(Z_1^{(k)})&=\int \varphi \circ f^{n-1}d\mu + \int  \chi^{(k)} \circ f^{n-1}d\mu-\int \chi^{(k)} \circ f^{n}d\mu-\int P^k\varphi \circ f^{n-1}d\mu\\
&=\int \varphi d\mu + \int  \chi^{(k)} d\mu-\int \chi^{(k)}d\mu-\int P^k \varphi d\mu=0.
\end{align*}
Hence, it remains to show that \( \E(Z_{j}^{(k)}|\mathcal F_{j-1}) = 0 \)  for every \( j=2,.., n \).  Using (P3) we have
\begin{align}
P\xi^{(k)}&=P\varphi+P\chi^{(k)}-PU\chi^{(k)}-P^{k+1}\varphi \nonumber
\\
& = P\varphi+P\chi^{(k)}-\chi^{(k)}-P^{k+1}\varphi
\nonumber
\\ &=P\varphi+\left(\sum_{n=1}^k P^{n+1}\varphi-\sum_{n=1}^k P^{n}\varphi\right)-P^{k+1}\varphi
\nonumber
\\&
=P\varphi+(P^{k+1}\varphi-P\varphi)-P^{k+1}\varphi
=0
\label{P0}
\end{align}
By property (P4) we have $\E(\cdot|f^{-(i+1)}(\mathcal M))=U^{i+1}P^{i+1}$, then
using property (P3) and \eqref{P0} it follows that for all $i=0,\ldots,n-2$,
\begin{equation}\label{mart}
\begin{aligned}
\E(Z_{n-i}^{(k)}|\mathcal F_{n-i-1}) &=
\E(\xi^{(k)}\circ f^i|f^{-(i+1)}(\mathcal M))
=U^{i+1}P^{i+1}U^i\xi^{(k)}
=U^{i+1}P\xi^{(k)}=0,
\end{aligned}
\end{equation}
which completes the proof that \( \{Z^{(k)}_{j}\}_{j=1}^n \) is a sequence of martingale differences.
\end{proof}

\begin{lemma}\label{lem:Pjq}
For any \( j\in\mathbb N \), \( q\geq 1 \) and  $\psi=\mbox{sgn} (P^j\varphi)$ we have
\begin{equation*}
\|P^j\varphi\|_q \leq
\cv_{\mu}(\varphi, \psi\circ f^{j})^{1/q}\|\varphi\|_{\B}^{1/q}\|\varphi\|_\infty^{1-1/q}.
\end{equation*}
\end{lemma}

\begin{proof}
We start by writing
\begin{align}
\|P^j\varphi\|_q &=\left(\int |P^j\varphi|^q\,d \mu\right)^{1/q}\nonumber
\\
&\leq \left(\|P^j\varphi\|_\infty^{q-1}\int |P^j\varphi|\,d\mu\right)^{1/q}\nonumber \\
&=\left(\|P^j\varphi\|_\infty^{q-1}\|P^j\varphi\|_1\right)^{1/q}. \label{eq:Pjq}
 \end{align}
We use  property (P5) to get
\[
 \|P^j\varphi\|_\infty^{q-1} \leq \|\varphi\|_\infty^{q-1}.
  \]
Then, taking $\psi=\mbox{sgn} (P^j\varphi)$, using
 property~(P2)
and  our assumptions on polynomial decay of correlations we have
\[
\|P^j\varphi\|_1=\int \left|P^j\varphi\right|  d\mu=
\int (P^j\varphi)\psi   d\mu=
\int \varphi (\psi \circ f^{n}) d\mu
= \|\varphi\|_{\B} \|\psi\|_\infty  \cv_{\mu}(\varphi, \psi\circ f^{j})
\]
 Thus, substituting into \eqref{eq:Pjq} and using that \( \|\psi\|_{\infty}=1 \) we get
\[
\|P^j\varphi\|_q = \left(\|P^j\varphi\|_\infty^{q-1}\|P^j\varphi\|_1\right)^{1/q} \leq
C^{1/q} \|\varphi\|_{\B}^{1/q} \|\varphi\|_\infty^{1-1/q}  j^{-\beta/q}.
\]
\end{proof}

\subsection{Polynomial case}\label{subsec:polynomial-DC=>LD}

\begin{proposition}\label{lem:poly}
 Let $\beta, C>0$ be such that  for all $\psi\in L^\infty$ we have
 \[
 \cv_\mu(\varphi,\psi\circ f^n) \leq C n^{-\beta}.
 \]
 Then there exists a constant \( C'>0 \) (depending only on \( C \)) such that  for every \( \epsilon>0 \) and \(  q>\max\{1,\beta\} \) we have
 \[
 \ld_{\mu}(\varphi,
\epsilon,n) \leq  C'\|\varphi\|_{\B}\|\varphi\|_{\infty}^{2q - 1}\epsilon^{-2q}  \,n^{-\beta}.
 \]
\end{proposition}

\begin{proof}
In order to apply Rio's inequality, we observe that for all $1\le i\le u\le n$
\begin{equation}\label{presum}
 \left\|X_{i}\sum_{j=i}^u \E(X_j|\mathcal F_i)\right\|_{q}
\leq \|X_{i}\|_{\infty} \left\|\sum_{j=i}^u \E(X_j|\mathcal F_i)\right\|_{q}
\leq \|\varphi \|_{\infty} \left\|\sum_{j=i}^u \E(X_j|\mathcal F_i)\right\|_{q}
 \end{equation}
and then,
using the standard fact that the conditional expectation of a sum of random variables is equal to the sum of the conditional expectations,  \( \|\sum_{j=i}^u \E(X_j|\mathcal F_i)\|_{q} \) is equal to
\begin{equation}\label{sum}
\left\|\sum_{j=i}^u \E(Z_j^{(k)}|\mathcal F_i)+\E(\chi^{(k)}\circ f^{n-i+1}-\chi^{(k)}\circ f^{n-u}|\mathcal F_i)+ \sum_{j=i}^u \E(P^k\varphi\circ f^{n-j}|\mathcal F_i)\right\|_{q}
\end{equation}
By the triangle inequality it is of course sufficient to bound the \( L^{q} \) norm of each term in the sum \eqref{sum}.
First we observe that from Lemma \ref{lem:Pjq} and using the assumption of polynomial decay of correlations we have
\begin{equation}\label{Pjnorm}
\|P^j\varphi\|_q \leq
\cv_{\mu}(\varphi, \psi\circ f^{j})\|\varphi\|_{\B}^{1/q} \|\varphi\|_\infty^{1-1/q}
\leq C^{1/q} \|\varphi\|_{\B}^{1/q}
\|\varphi\|_\infty^{1-1/q} j^{-\beta/q}.
\end{equation}
Therefore, recalling that  \( q>\beta \) and summing over \( j \), we get
\begin{equation}\label{chinorm}
\|\chi^{(k)}\|_{q}
= \left\|\sum_{j=1}^k P^j\varphi\right\|_{q}  \leq
 C^{1/q} \|\varphi\|_{\B}^{1/q}
\|\varphi\|_\infty^{1-1/q} \frac{  k^{1-\beta/q}}{1-\beta/q}
\end{equation}
Using \eqref{Pjnorm} and \eqref{chinorm} we can now obtain bounds for the three terms in \eqref{sum}. To simplify the notation we let \( \tilde C_{\varphi} = C^{1/q} \|\varphi\|_{\B}^{1/q} \|\varphi\|_\infty^{1-1/q} \). Recall that  for any \( q\geq 1 \), any random variable \( X \in L^{q} \) measurable with respect to a  \( \sigma \)-algebra \( \mathcal F \),
 and any sub-\( \sigma \)-algebra
\( \tilde{\mathcal F} \) of \( \mathcal F \)
we have \( \|\E(X|\tilde{\mathcal F})\|_{q} \leq  \|X\|_{q} \).
For the first term we showed in Lemma \ref{lem:mart} that the sequence \( \{Z^{(k)}_{j}\}_{j=1}^n \) is a martingale difference and  thus, in particular, for all $j=1,\ldots,n$,
we have $\E(Z_j^{(k)}|\mathcal F_{j-1})=0$ and therefore
\begin{equation}\label{firstterm}
\begin{aligned}
\left\|\sum_{j=i}^u \E(Z_j^{(k)}|\mathcal F_i)\right\|_{q}
&= \|\E(Z_i^{(k)}|\mathcal F_i)\|_{q}
\leq \|Z_{i}^{(k)}\|_{q}
= \|\xi^{(k)}\|_{q}
\\ &\leq  \|\varphi\|_{q} + 2\|\chi^{(k)}\|_{q}+ \|P^{k}\varphi\|_{q}
\\ &\leq \|\varphi\|_{q} +
\tilde C_{\varphi}
\left(\frac{2 k^{1-\beta/q}}{1-\beta/q}
+  k^{-\beta/q}\right).
\end{aligned}
\end{equation}
For the second   term we have
\begin{equation}\label{secondterm}
\begin{aligned}
 \|\E(\chi^{(k)}\circ f^{n-i+1}-\chi^{(k)}\circ f^{n-u}|\mathcal F_i)\|_{q}
&\leq \|\chi^{(k)}\circ f^{n-i+1}-\chi^{(k)}\circ f^{n-u}\|_{q}
\\ &\leq 2\|\chi^{(k)}\|_{q}\\
&\leq
\frac{2 \tilde C_{\varphi} k^{1-\beta/q}}{1-\beta/q}.
\end{aligned}
\end{equation}
Finally, for the third term, using that \( u\leq n \), we  have
\begin{equation}\label{thirdterm}
\begin{aligned}
\left\|\sum_{j=i}^u \E(P^k\varphi\circ f^{n-j}|\mathcal F_i)\right\|_{q}
&\leq
\sum_{j=i}^u \| \E(P^k\varphi\circ f^{n-j}|\mathcal F_i)\|_{q}
\\ &
\leq \sum_{j=i}^u \|P^k\varphi\|_{q}
 \leq n \|P^{k}\varphi\|_{q}
\leq  \tilde C_{\varphi} n k^{-\beta/q}.
 \end{aligned}
 \end{equation}
Substituting \eqref{firstterm}, \eqref{secondterm}, and \eqref{thirdterm} into
\eqref{sum} and then into \eqref{presum} we get
\[
\left\|X_{i}\sum_{j=i}^u \E(X_j|\mathcal F_i)\right\|_{q} \leq
\|\varphi\|_{\infty} \left(\|\varphi\|_{q} +
\tilde C_{\varphi}
\left(\frac{4 k^{1-\beta/q}}{1-\beta/q}
+  (n+1) k^{-\beta/q}\right)\right)
 \]
Taking \( k=n \), using that
\( \tilde C_{\varphi} = C^{1/q} \|\varphi\|_{\B}^{1/q} \|\varphi\|_\infty^{1-1/q} \) and applying Rio's inequality we get
\[
\|S_{n}\|_{2q}^{2q} \leq
\left(4q n \|\varphi\|_{\infty} \left(\|\varphi\|_{q} +
C^{1/q} \|\varphi\|_{\B}^{1/q} \|\varphi\|_\infty^{1-1/q}
\left(\frac{4 n^{1-\beta/q}}{1-\beta/q}
+  (n+1) n^{-\beta/q}\right)\right)\right)^{q}.
 \]
 To simplify this expression we use the notation \( \lesssim \) to indicate that the quantity on the left is less than the quantity on the right for sufficiently large \( n \) and up to some constant that may depend on \( \beta \) and \( q \) but not on \( \varphi \) (though the meaning of ``sufficiently large \( n \)'' is allowed to depend on \( \varphi \)). Then, for the innermost parenthesis,  taking into account that \( q>\beta \), we have
 \( \frac{4 n^{1-\beta/q}}{1-\beta/q}
+  (n+1) n^{-\beta/q} \lesssim n^{1-\beta/q} \).
The quantity \( n^{1-\beta/q} \)  is increasing with \( n \) and therefore the contribution of \( \|\varphi\|_{q} \) is negligible for sufficiently large \( n \) and so the middle parenthesis is \( \lesssim \|\varphi\|_{\B}^{1/q}\|\varphi\|_\infty^{1-1/q}n^{1-\beta/q} \) and  the outer parenthesis is
\( \lesssim   \|\varphi\|_{\B}^{1/q}\|\varphi\|_\infty^{2-1/q}n^{2-\beta/q}\). Therefore
\[
\|S_{n}\|_{2q}^{2q} \lesssim \|\varphi\|_{\B}\|\varphi\|_\infty^{2q-1}n^{2q-\beta}
 \]
Finally we  apply
the Markov-Chebyshev inequality
to get
\[
\mu\left(\frac1n\left|S_{n}\right|\geq \epsilon \right) \leq
\frac{\|S_n\|_{2q}^{2q} }{\epsilon^{2q}\,n^{2q}} \lesssim
\frac{\|\varphi\|_{\B}\|\varphi\|_\infty^{2q-1}n^{2q-\beta}}
{\epsilon^{2q}\,n^{2q}}=
\frac{ \|\varphi\|_{\B}\|\varphi\|_\infty^{2q-1}n^{-\beta}}
{\epsilon^{2q} }
\]
and this completes the proof in the polynomial case.
\end{proof}

\subsection{Stretched exponential case}\label{subsec:subexponential-DC=>LD}

\begin{proposition}\label{lem:strexp}
Let
$C,\tau,\theta>0$ be such that for all
$\psi\in L^\infty$ we have
 \[
 \cv_\mu(\varphi,\psi\circ f^n)\leq C e^{-\tau n^\theta}.
 \]
Then, for every \( \epsilon>0 \) and \(
 \tau'=\min\{\tau, {\epsilon^{2}}/{(162\|\varphi\|^2_{\infty})}\}
 \) we have
 \[
 \ld_{\mu}(\varphi,\epsilon,n)
 \leq \left(2+\frac {C \|\varphi\|_{\B}}{\epsilon}\right)
 e^{-\tau' n^{{\theta}/{(\theta+2)}}}.
 \]
\end{proposition}

\begin{proof}From \eqref{Sn1} we can bound  \( \mu\left(n^{-1}|S_n|>\epsilon\right) \) by
\begin{equation}\label{eq:subexp-large-dev}
 \mu\left(\frac{1}{n}\left|\sum_{j=1}^n  Z_j^{(k)}\right|> \frac{\epsilon}{3}\right)
+
 \mu\left(\frac{1}{n} |\chi^{(k)}\circ  f^n- \chi^{(k)}| >\frac{\epsilon}{3}\right)
 +\mu\left(\frac{1}{n}\left|\sum_{j=1}^n P^k\varphi\circ f^{n-j}\right|>\frac{\epsilon}{3}\right).
 \end{equation}
We shall estimate each of the three terms in \eqref{eq:subexp-large-dev} separately and by distinct arguments. We start with a preliminary remark which will be used for both the first and the second terms.
Since $P$ is defined with respect to the invariant measure $\mu$,  by property
(P5) we have that $\|P\varphi\|_\infty\leq \|\varphi\|_\infty$
and therefore we get
\(  \|\chi^{(k)}\|_\infty\leq k\|\varphi\|_\infty \)
which immediately implies
\begin{equation} \label{eq:chi-infty-norm}
\| \chi^{(k)}\circ f^n-\chi^{(k)}\|_{\infty} \leq 2k \|\varphi\|_{\infty}.
\end{equation}
From the  definition  of \( Z^{(k)}_{j} \) and
using \eqref{eq:chi-infty-norm}, we have for $k>2$ \begin{equation}\label{eq:z-infty-norm}
\|Z^{(k)}_j\|_\infty
\leq \|\varphi\|_\infty+2k\|\varphi\|_\infty+\|\varphi\|_\infty\leq 2(k+1)\|\varphi\|_\infty\leq 3k\|\varphi\|_\infty.
\end{equation}
By Lemma \ref{lem:mart} we know that the \( Z^{(k)}_{j}  \) form a sequence of martingale differences. Then,
letting \( b= \epsilon/3 \) and \( a = 3k\|\varphi\|_\infty \) and applying the Azuma-Hoeffding inequality  thus gives
\begin{equation}\label{lem:LDformartingales}
   \mu\left(\frac{1}{n}\left|\sum_{j=1}^n Z_j^{(k)}\right|>\frac{\epsilon}{3}\right)
   \leq 2\exp\left\{-\frac{n \epsilon^2}{162 k^2 \|\varphi\|^2_\infty}\right\}.
\end{equation}
To estimate the third term in \eqref{eq:subexp-large-dev} we
use  Chebyshev-Markov's inequality and
the invariance of $\mu$ to get
\begin{align}
\nonumber
 \mu\left(\frac{1}{n}\left|\sum_{j=1}^nP^k \varphi\circ f^{n-j}\right|>\frac{\epsilon}{3}\right)
 \nonumber
 &\leq \frac{3}{\epsilon n}\int\left|\sum_{j=1}^nP^k\varphi\circ f^{n-j}\right|d\mu\\
 \nonumber
 &\leq \frac{3}{\epsilon n}\sum_{j=1}^n\int\left|P^k\varphi\circ f^{n-j}\right|d\mu\\
 \nonumber
 &\leq \frac{3}{\epsilon}\int\left|P^k\varphi\right|d\mu
 \\
 & \leq \frac{3}{\epsilon}C \|\varphi\|_{\B}\e^{-\tau k^\theta}.
 \label{LDformartingales}
\end{align}
For the last inequality we have used a  simple application of Lemma \ref{lem:Pjq} with \( q=1 \) and our assumptions on the stretched exponential decay of correlations.
Notice that the estimates obtained in \eqref{lem:LDformartingales} and
\eqref{LDformartingales}  involve  \( k \).
At this point we set $$k=k(n):=n^{1/(\theta+2)}.$$ Then, for all sufficiently large \( n \),  we have from \eqref{eq:chi-infty-norm} that the condition in the second term of \eqref{eq:subexp-large-dev}
is never satisfied and so the term vanishes. Therefore substituting
\eqref{lem:LDformartingales} and
\eqref{LDformartingales} and the formula for \( k(n) \)
into \eqref{eq:subexp-large-dev}
we get
\[
\mu\left(\frac{1}{n}|S_n|>\epsilon\right)
\leq 2\exp\left\{-\frac{\epsilon^2}{168\|\varphi\|^2_\infty}\; n^{\frac\theta{\theta+2}}\right\}+ \frac{C\|\varphi\|_{\B}}{\epsilon}\;\exp\left\{-\tau n^{\frac\theta{\theta+2}}\right\}.
\]
This completes the proof of Proposition \ref{lem:strexp}.
\end{proof}

\subsection{Exponential case}\label{subsec:exponential-DC=>LD}
To deal with the exponential case, we start with a preliminary result which essentially uses the duality of $L^1$ and $L^\infty$.

\begin{lemma}\label{lem:sum}
Let $\varphi\in L^\infty$ with \( \int \varphi d\mu =0 \). If there is $\xi(n)$ with $\sum_{n=0}^\infty\xi(n)<\infty$ and
$
  \cv_\mu(\varphi,\psi\circ f^n) \leq\xi(n)
$ for all $\psi\in L^1$, then
 $$
  \sum_{n=0}^{\infty}P^{n}\varphi\in L^{\infty}.
 $$
\end{lemma}

\begin{proof}

By Riesz' representation theorem we may identify $L^\infty$ with the dual of $L^1$ by associating to $\varphi\in L^\infty$ the linear functional $\ell_\varphi:L^1\to \R$ defined by $\ell_\varphi(\psi)=\int \varphi \psi d\mu$. Since $\|\varphi\|_\infty=\|\ell_\varphi\|$, we have for all $n\ge0$
\begin{align*}
   \|P^{n}\varphi\|_{L^{\infty}}
 &  =\sup_{\psi\in L^1} \frac{|\int (P^{n}\varphi )\psi d\mu|}{\|\psi\|_{L^1}}
 \\ &=\sup_{\psi\in L^1} \frac{|\int \varphi (\psi\circ f^n )d\mu|}{\|\psi\|_{L^1}}
\\ &= \frac{\|\varphi\|_{\infty}\|\psi\|_{L^1} \cv_\mu(\varphi,\psi\circ f^n) }{\|\psi\|_{L^1}}
\\ &\le\|\varphi\|_{\infty} \xi(n).
\end{align*}
Therefore
\[
\left\|\sum_{n=0}^{\infty}P^{n}\varphi \right\|_{\infty} \leq \sum_{n=0}^{\infty} \|P^{n}\varphi\|_{\infty}
\leq \|\varphi\|_{\infty}\sum_{n=0}^{\infty} \xi(n) <\infty.
\]
\end{proof}

\begin{proposition}
Let \( \varphi \in L^{\infty} \) and suppose that
\[
 \sum_{n=0}^{\infty} P^{n}\varphi \in L^{\infty}.
  \]
  Then  for every \( \epsilon > 0 \) there exists  \( C'=C'({\varphi, \epsilon})> 0 \) such that
 \[
 \ld_{\mu}(\varphi,\epsilon,n) \leq C' e^{- \tau n},
 \]
 where \(\tau=  1/8(\|\varphi\|_{\infty}+2\|\sum P^{n}\varphi\|_{\infty})^{2} \).
\end{proposition}

\begin{proof}

We show that
  $\{Z_j\}_{j=1}^n$ as defined in \eqref{eq:Zj} is a finite sequence of martingale differences with respect to the filtration
 $\{\mathcal F_j\}_{j=1}^n$, where $\mathcal F_j= f^{-(n-j)}\mathcal M$, as in \eqref{eq:def-filtration-special}.
Indeed, as before, we also have that $Z_j$ is measurable with respect to $\mathcal F_j$ for all $j=1,\ldots,n$ and
\begin{equation*}
\E(Z_1)=\int \varphi \circ f^{n-1}d\mu + \int  \chi\circ f^{n-1}d\mu-\int \chi \circ f^{n}d\mu=\int \varphi d\mu + \int  \chi d\mu-\int \chi d\mu=0.
\end{equation*}
Furthermore
\[ P\xi=P\varphi+P\chi-PU\chi =P\varphi+(P\chi-\chi) =P\varphi-P\varphi=0,
\]
which allows us to conclude that for all $i=0,\ldots, n-2$,
\begin{equation*}
\begin{aligned}
\E(Z_{n-i}|\mathcal F_{n-i-1}) &=
\E(\xi\circ f^i|f^{-(i+1)}(\mathcal M))
=U^{i+1}P^{i+1}U^i\xi
=U^{i+1}P\xi=0,
\end{aligned}
\end{equation*}
where we used property (P3) and the fact that property (P4) implies that
 $$\E(\cdot|f^{-(i+1)}(\mathcal M))=U^{i+1}P^{i+1}.$$
Additionally, from the definition of \( Z_{j} \)
we have, for all $j=1,\ldots,n$,
\begin{equation}\label{eq:z-infty-norm-2}\|Z_j\|_\infty\leq \|\varphi\|_\infty+2\|\chi\|_\infty.
\end{equation}
and therefore, by the Azuma-Hoeffding inequality we get
\[
\mu\left(\frac{1}{n}\left|\sum_{j=1}^n Z_j\right|>\frac{\epsilon}{2}\right)
\leq 2\exp\left\{-\frac{\epsilon^2}{8(\|\varphi\|_\infty+2\|\chi\|_\infty)^2 }\;n\right\}.
\]
Thus, for all sufficiently large values of \( n \), in particular for \( n\geq N \)
where $2/N\|\chi\|_\infty\leq\epsilon/2$ we have

\begin{align*}
\mu\left(\frac{1}{n}|S_n| >\epsilon\right) &\leq
 \mu\left(\frac{1}{n}\left|\sum_{j=1}^n Z_j\right|+\frac{2}{n}\|\chi\|_\infty>\epsilon\right) \\ &\leq
   \mu\left(\frac{1}{n}\left|\sum_{j=1}^n Z_j\right|>\frac{\epsilon}{2}\right)
   \leq 2\exp\left\{-\frac{\epsilon^2}{8(\|\varphi\|_\infty+2\|\chi\|_\infty)^2 }\;n\right\}
\end{align*}
\end{proof}

\section{Gibbs-Markov structures for local diffeomorphisms}
\label{sec:locdif}

In this section we prove Theorem \ref{th:corr=>tower}.
We consider the function
\[ \phi (x) := \log \|Df(x)^{-1}\|
\]
and note that
 in the case of \( C^{1+} \) local diffeomorphisms, \( \phi  \) is H\"older continuous.
 From the assumptions of Theorem \ref{th:corr=>tower} and the conclusions of Theorem
 \ref{th:DC=>LD} we therefore have
 \begin{equation}\label{eg:ld}
 LD_{\mu}(\phi, \epsilon, n)= \mathcal O(n^{-\beta})
 \quad
 \text{ and }
 \quad
  LD_{\mu}(\phi, \epsilon, n)= \mathcal O(e^{-\tau n^{\theta}})
\end{equation}
in the polynomial case and in
 the stretched  and exponential cases respectively.
 Theorem~\ref{th:corr=>tower} then follows directly from

\begin{theorem}\label{th:LD=>tower}
Let \( f \) be a \(
C^{1+} \)  local diffeomorphism with an ergodic expanding acip \( \mu \);
\begin{enumerate}
\item
if there exists \( \beta > 1\) such that for  small $\epsilon>0$ we have
$\ld_{\mu}(\phi,\epsilon,n)\lesssim n^{-\beta}$, then there is a Gibbs-Markov induced map with
$\m(\mathscr R_{n}) \lesssim n^{-\beta+1}$.
\end{enumerate}
Suppose moreover that  \( d\mu/d\l \) is uniformly bounded away from 0 on its support. Then
\begin{enumerate}
\item[(2)] if there exist \( \tau, \theta > 0 \) such that for small
 \( \epsilon > 0  \) we have
$\ld_{\mu}(\phi,\epsilon,n)\lesssim e^{-\tau n^{\theta}}$, then there is a Gibbs-Markov induced map with $\m(\mathscr R_{n}) \lesssim e^{-\tau'n^{\theta}}, $ for some~$\tau'>~0$.
\end{enumerate}
\end{theorem}
Notice that the second part of the theorem applies in particular if \( \theta=1 \), i.e. in the exponential case. Notice also that the large deviation rates are not assumed to be uniform in \( \epsilon \).
To prove this theorem we first state a general result  which will also be useful in the case of maps with critical/singular sets.
Suppose we are given an arbitrary function
 \( \varphi \in L^{1} \).  Define
\[ \tilde S_{n}:= \left|\frac 1 n
\sum_{i=0}^{n-1}\varphi\left(f^i(x)\right)-\int\varphi
d\mu\right|.
\]
Then \( \tilde S_{n}(x) \to 0 \) for \( \mu \) almost every \( x \).  Notice that the large deviation estimates are precisely bounds on the rate of decay of the tail \( \mu\{\tilde S_{n} >  \epsilon\} \).
For \( \epsilon > 0 \) define
\begin{equation}\label{eq:Neps}
  N_{\epsilon}(x):= \min\{N: \tilde S_{n}\leq \epsilon \quad \forall \ n\geq N\}.
 \end{equation}

\begin{lemma}\label{lem:ldtail}
 Let \( A\subseteq M \) be such that \(d\mu/dm > c \) on \( A \) for some \( c>0 \). Suppose that given \( \varphi\in L^{1} \) and \( \epsilon > 0 \)  there exists  \( \xi: \mathbb N \to \mathbb R^{+} \) such that
\(
 LD_{\mu}(\varphi, \epsilon, n) \leq \xi(n).   \)
Then for every \( n\geq 1 \) we have
\[
 m(\{N_{\epsilon}>n\}\cap A) \leq \frac 1c\sum_{\ell \geq n} \xi (\ell).
   \]
\end{lemma}

\begin{proof}
For  \( \epsilon >0  \)  we have
\[
\{ N_{\epsilon}>n\}\subset  M\setminus \bigcap_{\ell\geq n}\{\tilde S_{\ell}\leq \epsilon\} \subset \bigcup_{\ell \geq n} \{\tilde S_{n} > \epsilon\}.
 \]
The assumption on the density gives
 \( m(B) \leq
\|dm/d\mu\|_{\infty} \mu(B)\leq \mu(B)/c \)
for any
measurable set \( B\subset A \), and therefore
\[
m(\{ N_{\epsilon}>n\}\cap A)
\leq  \frac 1c\mu (\{ N_{\epsilon}>n\}\cap A)
  \leq  \frac 1c\mu ( \bigcup_{\ell \geq n} \{\tilde S_{\ell}\geq \epsilon\}) \leq \frac 1c\sum_{\ell\geq  n} \xi (n).
 \]
  The last inequality uses the
  assumption on the large deviation rate function which gives
\( \mu \{\tilde S_{n}\geq \epsilon\}\leq \xi (n) \).
\end{proof}

\begin{proof}[Proof of Theorem
\ref{th:LD=>tower}]
By the expansivity assumption on \( \mu \)
and  a straightforward application of Birkhoff's ergodic theorem we have that \begin{equation}\label{non-unif-expand}
   \lim_{n\to\infty}\frac{1}{n}\sum_{i=0}^{n-1}
   \phi (f^{j}(x)) =\int \phi d\mu =: \lambda < 0
\end{equation}
is satisfied \( \mu \) almost everywhere.
Thus we have that
\[
\mathcal E(x) :=  \min\left\{N: \frac{1}{n} \sum_{i=0}^{n-1}
\phi (f^{j}(x)) \leq  \lambda/2 \ \ \forall n\geq
N\right\}.
\]
 is defined and finite almost everywhere in \( M \).
  Notice that using the notation in \eqref{eq:Neps} for  \( \varphi = \phi \) and \( \epsilon = \lambda/2 \) we have that
 \[ \{\mathcal E>n\}\subseteq \{N_{\epsilon}>n\}.
  \]

In \cite{ALP, Gou} induced Markov maps are constructed and it is shown that that tails of return times have the same rate of decay (polynomial, stretched or exponential)  as the rate of decay of the Lebesgue measure of \( m\{\mathcal E > n \}
 \). The conclusion therefore follows  from an application of Lemma \ref{lem:ldtail},
 substituting the corresponding polynomial or (stretched) exponential bounds. We just need to specify the set A on which the density of \( \mu \) is bounded below.

 For the  polynomial case we take advantage of a result of \cite{AlvDiaLuz}  where it is shown that there exists a ball \( \Delta\subset \operatorname{supp}  (\mu) \)
centred at a point \( p \) whose preimages are dense in the support of \( \mu \), such that
the  density of \( \mu  \) with respect to Lebesgue is uniformly bounded below on \( \Delta \). This is sufficient for the purposes of applying the construction of \cite{ALP} which indeed only requires the existence of such a ball and where the required tail estimates are then formulated in terms of the decay of \( m(\{\mathcal E>n\}\cap \Delta \)). In the stretched and exponential case we apply instead the arguments of \cite{Gou} which which rely on  somewhat more global assumptions and therefore require a control of the density on the entire support of \( \mu \). For this reason
we need to include the boundedness from below of the density
as part of our assumptions.
Theorem \ref{th:LD=>tower}
 is now a direct consequence of the following where we let \( A=\Delta \) in the polynomial case, or \( A=\operatorname{supp}(\mu) \) in the other cases.
\end{proof}

\section{Gibbs-Markov structures for maps with critical/singular sets}
\label{sec:unbdev}

In this section we consider maps with critical/singular sets and prove Theorem \ref{th:CS-corr=>tower}.
We shall follow a similar strategy used in the proof of Theorem \ref{th:corr=>tower}
and once again we aim to apply the construction and estimates of \cite{ALP, Gou}. A main difference here is that the function \( \log \|Df^{-1}\| \) is not necessarily H\"older continuous and therefore we cannot apply directly the results of Theorem
\ref{th:DC=>LD} which give bounds on the large deviation rates. Moreover, we also need to consider an additional function related to the recurrence to the critical/singular set.
We let
\begin{equation*}
\phi_{1}(x) = \log \|Df^{-1}\| \quad \text{ and }
\quad
\phi_{2}(x)=\phi_{2}^{(\delta)}(x)=\begin{cases}
-\log\dist(x,\mathcal C)&\text{ if $\dist(x,\mathcal C)<\delta$ },\\
\frac{\log\delta}{\delta}(\dist(x,\mathcal C)-2\delta)&\text{ if $\delta\leq \dist(x,\mathcal C)<2\delta$ },\\
0&\text{ if $\dist(x,\mathcal C)\geq2\delta$},
\end{cases}
\end{equation*}
where $\delta>0$ is a small constant to be fixed later.  We remark that \( \phi_{2}(x)= -\log \dist(x,\mathcal C) \) in the \( \delta \) neighbourhood  and \( \phi_{2}(x)=0 \) outside a \( 2\delta \) neighbourhood of the critical set \( \mathcal C \). The definition in the remaining region is motivated by the requirement that the function be H\"older continuous except at the critical/singular set. We do need some large deviation estimates for these functions as we had in \eqref{eg:ld} for the local diffeomorphism case. These are provided in the following

\begin{proposition}
\label{prop:DC=>LD-nonempty-CS}
Let \( f: M\to M \) be a \(
C^{1+} \)  local diffeomorphism
outside
a nondegenerate critical set \( \mathscr C \).
Suppose that
\( f \) admits an ergodic expanding acip \( \mu \) with
\( d\mu/d\l \in L^{p}(m) \) for some \( p>1 \);
\begin{enumerate}
\item
if
there exists \( \beta > 1\) such that
\( \cv_{\mu} (\varphi, \psi \circ f^{n}) \lesssim n^{-\beta}
\)   for every \( \varphi \in \mathcal H \) and \( \psi \in L^{\infty} \),
then there is  $C'>0$ such that
\(\ld_{\mu}(\phi_i, \epsilon,n)\leq C'\,n^{-\beta+\gamma}\), for \( i=1,2 \).
\end{enumerate}
Suppose moreover that  \( d\mu/d\l \) is uniformly bounded away from 0 on its support;
\begin{enumerate}
 \item[(2)]
if there exist \( \tau, \theta > 0 \)
such that  \(
\cv_{\mu} (\varphi, \psi \circ f^{n}) \lesssim e^{-\tau n^{\theta}}\)
for every \( \varphi \in \mathcal H \) and \( \psi \in L^{\infty} \), then
there exists \( \zeta>0 \) such that for any \( \gamma>0 \) and \( \epsilon>0 \) sufficiently small there is $C'>0$ such that
\(
 \ld_{\mu}(\phi_i, \epsilon,n) \leq
C' \e^{-\zeta n^{\theta/3-\gamma }}
\) for \( i=1, 2 \).
\end{enumerate}
\end{proposition}

The proof of Proposition \ref{prop:DC=>LD-nonempty-CS} is relatively technical and we postpone it to the following section. Assuming the conclusions of this proposition for the moment, Theorem \( B \) follows from

\begin{theorem}\label{th:LD=>tower2}
Let \( f: M\to M \) be a \(
C^{1+} \)  local diffeomorphism
outside
a nondegenerate critical set \( \mathscr C \).
Suppose that
\( f \) admits an ergodic expanding acip \( \mu \) with \( d\mu/d\l \) with
\( d\mu/d\l \in L^{p}(m) \) for some \( p>1 \).
Then \( \phi_{i}\in L^{1}(\mu) \) for \( i=1,2 \). Moreover,
\begin{enumerate}
\item
if there exists \( \beta > 1\) such that for  small $\epsilon>0$ we have
$\ld_{\mu}(\phi_{i},\epsilon,n)\lesssim n^{-\beta}$ for \( i=1,2 \), then there is a Gibbs-Markov induced map with
$\m(\mathscr R_{n}) \lesssim n^{-\beta+1}$.
\end{enumerate}
Suppose moreover that  \( d\mu/d\l \) is uniformly bounded away from 0 on its support;
\begin{enumerate}
\item[(2)] if there exist \( \tau, \theta > 0 \) such that for small
 \( \epsilon > 0  \) we have
$\ld_{\mu}(\phi_{i},\epsilon,n)\lesssim e^{-\tau n^{\theta}}$ for \( i=1,2  \), then there is a Gibbs-Markov induced map with $\m(\mathscr R_{n}) \lesssim e^{-\tau'n^{\theta}}, $ for some~$\tau'>~0$.
\end{enumerate}
\end{theorem}
Notice that the second part of the theorem applies in particular if \( \theta=1 \), i.e. in the exponential case. Notice also that the large deviation rates are not assumed to be uniform in \( \epsilon \).  We begin by introducing the natural auxiliary function
\[ \phi_{0}(x):= -\log \dist(x, \mathcal C).
 \]
Then, for \( i=0,1,2 \) and \( k> 0 \) we let
\[
 A_{i, k} := \{x: \phi_{i}(x) \geq k\}
\]

\begin{lemma} \label{lem:Aik}
There exists \( \zeta>0 \) such that
for all \( k >0  \) and for all \( i=0,1,2 \)
we have
\[
(1)\   \mu(A_{i, k}) \lesssim e^{-\zeta k };  \qquad
  (2) \ \phi_{i}\in L^{1}(\mu);  \quad \text{ and }
\quad (3)  \ \int \phi_{2}^{(\delta)}d\mu \to 0  \text{ as } \delta \to 0.
  \]
  \end{lemma}
\begin{proof}
Recall that we have assumed that \( d\mu/dm\in L^{p}(m) \) for some \( p>1 \). We define \( q>1 \) by the usual condition \( 1/p+1/q=1 \).
Then, by H\"older's inequality,  we have
\[
\mu(A_{0, k}) =
\int \mathbb 1_{A_{0, k}} \frac{d\mu}{dm} dm
\leq \|\mathbb 1_{A_{0, k}}\|_{q} \left\|\frac{d\mu}{dm}\right\|_{p}\leq
\m(A_{0, k})^{1/q}\left\|\frac{d\mu}{dm}\right\|_{p} \lesssim \m(A_{0, k})^{1/q}
\] and thus  (1) for \( i=0 \) follows directly from condition (C0). For \( i=2 \) we also have the result since \( \phi_{2}(x)=\phi_{0}(x) \)  as long as \( \dist (x,\mathcal C) \leq  \delta \) or, equivalently, \( k\geq -\log \delta \).
For \( i=1 \) we use condition (C1) which implies that  there exists a constant \( \tilde B>0 \) such that for every \( x\in M\setminus \mathcal C \) we have
\begin{equation}\label{eq:logdist}
-\tilde B+\eta\log d(x, \mathcal C) \leq \phi_{1}(x) \leq \tilde B-\eta\log d(x, \mathcal C).
 \end{equation}
 Therefore there exists some constant \( \tilde\eta>0 \) such that
\( \{\phi_{1}\geq k\}\subseteq \{\phi_{0}> \tilde\eta k\} \) which then clearly gives the conclusion for \( \phi_{1} \) and thus completes the proof of (1).
The integrability of \( \phi_{i} \) in (2) now follows easily from the fact that for \( i=0,1,2 \)
\[
 \int \phi_{i}d\mu \leq \sum_{n=1}^{\infty} \mu(A_{i,n})
  \]
and using (1). Finally, to prove (3) we let \( k_{1}=-\log \delta \), \( k_{2}=-\log 2\delta \) and write
\[
\int \phi_{2}^{(\delta)}d\mu =  \int_{A_{0, k_{1}}} \phi_{2}^{(\delta)}d\mu + \int_{A_{0, k_{2}\setminus A_{0, k_{1}}   }} \phi_{2}^{(\delta)}d\mu + \int_{M\setminus A_{0, k_{2}}} \phi_{2}^{(\delta)}d\mu
 \]
Since \( \phi_{2}^{(\delta)}(x) = 0 \) for \( x\in M\setminus A_{0, k_{2}} \), the third term vanishes. For the first term notice that \( \phi_{2}^{(\delta)}(x) = \phi_{0}(x) \) for \( x\in A_{0, k_{1}} \). Since \( \phi_{0}\in L^{1}(\mu) \) and \( \mu(A_{0, k_{1}}) \to 0 \) as \( \delta \to 0 \), it follows that
\[ \int_{A_{0, k_{1}}}\phi_{0} d\mu \to 0 \quad \text{ as } \quad  \delta \to 0.
  \]
Finally, for the middle term we have \( \phi_{2}^{(\delta)}(x) \leq -\log \delta \) for \( x\in A_{0, k_{2}\setminus A_{0, k_{1}}   } \), and so
\[
\int_{A_{0, k_{2}\setminus A_{0, k_{1}}   }} \phi_{2}^{(\delta)}d\mu
\leq (-\log \delta)  \mu(A_{0, k_{2}}) \lesssim (-\log \delta) e^{-\zeta k_{2}}
\leq (-\log \delta) (2\delta)^{\zeta}
 \]
 which clearly tends to zero as \( \delta \to 0 \).
  \end{proof}

\begin{proof}[Proof of Theorem \ref{th:LD=>tower2}]

We follow a similar strategy as in the proof of Theorem \ref{th:LD=>tower}, applying the results of \cite{ALP, Gou}. We consider as before the tail \( \{\mathcal E(x)>n\} \) of the expansion time related to the function \( \phi_{1} \) but also need to consider an analogous term related to the function \( \phi_{2} \).
More precisely we need to show that for \( \epsilon> 0 \) sufficiently small,
 there exists \( \delta> 0 \) such that
 \[     \limsup_{n\to+\infty}
\frac{1}{n} \sum_{j=0}^{n-1}- \log \dist_{\delta}(f^{j}(x), \mathcal C)
 \leq \epsilon. \]
 We note that  it is sufficient to have this for some \( \epsilon>0 \) depending only on the map, see e.g. \cite[Remark 3.8]{Alv}.
 In fact, fixing such an \( \epsilon \),  from Lemma \ref{lem:Aik}
 we can choose \( \delta>0 \) sufficiently small so that
\begin{equation} \label{slow-recurrence}
    \limsup_{n\to+\infty}
\frac{1}{n} \sum_{j=0}^{n-1}- \log \dist_{\delta}(f^{j}(x), \mathcal C)
\leq     \lim_{n\to+\infty}
\frac{1}{n} \sum_{j=0}^{n-1} \phi_{2}^{(\delta)}(f^{j}(x))
=
\int \phi_{2}^{(\delta)}  d\mu \leq \epsilon.
\end{equation}
We introduce the
\emph{recurrence time} function
\[
\mathcal R_{\epsilon,\delta}(x) = \min\left\{N\ge 1: \frac{1}{n} \sum_{i=0}^{n-1}
- \log \dist_{\delta}(f^{j}(x), \mathcal C)  \leq 2\epsilon, \ \ \forall
n\geq N\right\}
\]
which is defined and finite \( \mu \) almost everywhere in \( M \). Using again the notation in \eqref{eq:Neps} for \( \varphi= \phi_{2}^{(\delta)}\) we have
\[ \{\mathcal R_{\epsilon, \delta}>n\}\subseteq \{N_{\epsilon}>n\}.
 \]
In \cite{ALP, Gou} induced Markov maps are constructed and it is shown that that tails of return times have the same rate of decay (polynomial, stretched or exponential)  as the rate of decay of the Lebesgue measure of
\[
\{x: \mathcal E(x) > n \ \text{ or } \ \mathcal R_{\epsilon,\delta}(x) > n \}.
\]
The conclusion therefore follows  from an application of Lemma \ref{lem:ldtail},
 substituting the corresponding polynomial or (stretched) exponential bounds.
 We note that here we take \( A \) equal to the whole support of \( \mu \) since we have the density uniformly bounded below by assumption.
\end{proof}

\section{Large deviations for the special non h\"older observables}
\label{sec:nonhold}

In this section we prove Proposition \ref{prop:DC=>LD-nonempty-CS}.
Our strategy is to approximate \( \phi_{1} \) and \( \phi_{2} \)
by  ``truncated'' functions which are H\"older continuous.
For all \( k >0  \) and \( i=1,2 \), we let
\[ A_{i, k} := \{x: \phi_{i}(x) \geq k\}
\]
We now  define
\[
\phi_{i,k}(x):= \begin{cases}
\phi_{i}(x) &\text{ if } x\in  M\setminus A_{i, k}
\\
k &\text{ if } x\in A_{i,k}
\end{cases}
\]
Then we can write, for \( i=1,2 \) and \( n\in\mathbb N \)
\begin{align}
\mu\left(\frac{1}{n}|S_{n}\phi_{i}(x)|> \epsilon\right)
& \leq \mu \left(
\left\{\frac{1}{n}|S_{n}\phi_{i}(x)|> \epsilon\right\}
\setminus \bigcup_{j=0}^{n-1} f^{-j}A_{i, k}\right)
+ \mu\left(\bigcup_{j=0}^{n-1} f^{-j}(A_{i, k})\right)
\notag
\\&
\leq \mu\left(\frac{1}{n}|S_{n}\phi_{i, k}(x)| > \epsilon\right)
 + \sum_{j=0}^{n-1}\mu\left(f^{-j}(A_{i, k})\right)
 \notag
 \\&
 \leq \mu\left(\frac{1}{n}|S_{n}\phi_{i, k}(x)| > \epsilon\right)
 + n \mu(A_{i, k}).
  \label{ldder}
 \end{align}
The invariance of the measure \( \mu \) is used in the last step.
The second term in \eqref{ldder} is now easily bounded as in the following

To bound the first term of \eqref{ldder}, notice that
\( \phi_{i,k} \) is H\"older continuous with exponent \( \alpha \) for every \( \alpha\in (0, 1] \). Therefore we shall use our assumptions which apply to H\"older continuous observables, in particular we will apply the conclusions of Propositions \ref{lem:poly} and \ref{lem:strexp} with \( \B=\mathcal H_{\alpha} \).
 For this we need to obtain bounds for the \( L^{\infty} \) and H\"older norms of the the functions \( \phi_{i,k} \).

\begin{lemma} \label{lem:noname}
For any \( \alpha\in (0,1] \) and \( i=1,2 \) we have
\(
\|\phi_{i,k} \|_{\mathcal H_{\alpha}} \lesssim  \alpha^{-1} k e^{\alpha k}
\).
\end{lemma}
\begin{proof}
By the definition of \( \phi_{i, k} \) we have  \( \|\phi_{i, k}\|_{\infty} \leq k \) for \( i=1,2 \). Given \( x,y \in M\setminus\mathcal C \) and
 assuming without loss of generality that $\dist(y, \mathcal C)\geq\dist(x, \mathcal C)$ we have
 \begin{equation}\label{eq:phiik}
\frac{ \left|\log \dist(y, \mathcal C)-
\log \dist(x, \mathcal C) \:\right|}{\dist(x,y)^\alpha}\leq \frac{\left|
\log\left(1+\frac{\dist(y, \mathcal C)-\dist(x, \mathcal C)}{\dist(x, \mathcal C)}\right)\right|}{\dist(x, \mathcal C)^\alpha\left(\frac{\dist(x,y)}{\dist(x, \mathcal C)}\right)^\alpha}
\leq \frac{\left|
\log\left(1+\frac{\dist(x,y)}{\dist(x, \mathcal C)}\right)\right|}{\left(\frac{\dist(x,y)}{\dist(x, \mathcal C)}\right)^\alpha} \dist(x, \mathcal C)^{-\alpha}.
\end{equation}
Notice that the function  \( z^{-\alpha}\log(1+z) \) is bounded above with a global maximum  \( z_{0} \) satisfying \( \log (1+z_{0}) = z_{0}\alpha^{-1} (1+z_{0})^{-1} \). Substituting this back into the function we get
\( z_{0}^{-\alpha}\log(1+z_{0}) = \alpha^{-1} z_{0}^{1-\alpha}(1+z_{0})\)  which is bounded by \( 1/\alpha \).  Using this bound in \eqref{eq:phiik} we get, for \( x  \) such that    \( \dist (x, \mathcal C) \geq e^{-k} \),
\[
 \frac{\left|\log \dist(y, \mathcal C)-
\log \dist(x, \mathcal C) \right|}{\dist(x,y)^{\alpha}}\leq \frac 1\alpha\dist(x,\C)^{-\alpha} \leq \frac 1\alpha e^{\alpha k}.
\]
From   (C2) in the nondegeneracy conditions and using that  \( k+\alpha^{-1}e^{\alpha k }\lesssim \alpha^{-1}ke^{\alpha k} \) we thus obtain the required bound
for \( \phi_{1,k} \).

 For \( \phi_{2,k} \) we just need to consider the extra term corresponding to the region where both
\( d(x, \mathcal C)  \) and \( d(y, \mathcal C) \) belong to \((\delta, 2\delta) \). Here  we have
\[
\frac{|\phi_{2,k}(x)-\phi_{2,k}(y)|}{d(x, y)^{\alpha}}
\leq \frac{\log \delta}{\delta}\frac{|d(x, \mathcal C)-d(y, \mathcal C)|}{d(x,y)^{\alpha}} \leq \frac{\log \delta}{\delta} d(x,y)^{1-\alpha}\lesssim
\frac{\log \delta}{\delta}.
 \]
 Keeping in mind that \( \delta \) is fixed, this
  completes the proof for \( \phi_{2,k} \).
\end{proof}

\begin{proof}
We are now ready to estimate the first term in \eqref{ldder}.
From this point onwards, all estimates will apply equally to \( \phi_{1,k} \) and \( \phi_{2, k} \). Thus, to simplify the notation we shall just write \( \phi_{k} \).
We consider first of all the polynomial case.
Substituting the estimates of  Lemma \ref{lem:noname} into the  results of Proposition \ref{lem:poly} we get
\begin{equation}
\label{polyhold}
\mu\left(\frac{1}{n}|S_{n}\phi_{k}(x)| > \epsilon\right)
\lesssim\|\phi_{k}\|_{\mathcal H_\alpha}\|\phi_{k}\|_{\infty}^{2q - 1}\epsilon^{-2q}
 n^{-\beta}
   \lesssim k^{2q}e^{\alpha k}\,n^{-\beta},
 \end{equation}

Using Lemma \ref{lem:Aik} and substituting \eqref{polyhold} into \eqref{ldder} gives
\begin{equation}\label{sn1}
\mu\left(\frac{1}{n}|S_{n}\phi(x)|> \epsilon\right) \lesssim
k^{2q}e^{\alpha k} n^{-\beta} + n e^{-\zeta k}.
\end{equation}
We now complete the estimate by choosing  \( k \) appropriately and taking advantage of the fact that we can also choose \( \alpha \) arbitrarily small. Indeed, if \( \varphi\in \mathcal H_{\alpha'} \) then \( \varphi \in \mathcal H_{\alpha} \) for all \( \alpha\in (0, \alpha') \). We aim to obtain an upper bound of the order of \( n^{-\beta+\gamma} \) and thus require  that the two inequalities
\[
 n e^{-\zeta k}  \lesssim
 n^{-\beta+\gamma}\quad\text{ and } \quad
  k^{2q}e^{\alpha k}\lesssim n^{\gamma}
  \]
  aere simultaneously satisfied.
   We will show that this can be achieved by fixing a sufficiently small \( \alpha \) and then choosing \( k, n \) sufficiently large.
 First observe  that
 \[
k \geq \frac{\beta+1-\gamma}{\zeta} \log n \quad \Rightarrow \quad
 ne^{-\zeta k} \leq n^{-\beta+\gamma}
  \]
  and
  \[
  \alpha k + 2q\log k \leq \log \alpha  +  \gamma \log n
 \quad \Rightarrow \quad
 \frac 1\alpha k^{2q} e^{\alpha k} \leq n^{\gamma}.
     \]
 Now for any fixed \( \alpha \) and \( k=k(\alpha) \) sufficiently large, we have
 \( \alpha k + 2q\log k \leq 2\alpha k \); also for \( n=n(\alpha)  \) sufficiently large we have \( \log\alpha + \gamma\log n \leq \frac \gamma 2 \log n \).
 Therefore we can write the one-sided implication
  \[  k \leq \frac{\gamma}{2\alpha}\log n \quad \Rightarrow \quad
 \frac 1\alpha k^{2q} e^{\alpha k} \leq n^{\gamma}
  \]
Thus it is enough to show that for \( \alpha  \) sufficiently small we have
   \[
   \frac{\beta+1-\gamma}{\zeta} \log n \leq \frac{\gamma}{2\alpha}\log n.
    \]
This is clearly true and in fact we can choose the explicit value
 \[
 \alpha= \frac{\gamma \zeta}{2(\beta + 1 - \gamma)}.
  \]
 This completes the proof in the polynomial case.

We now consider the stretched exponential case. Substituting the estimates of Proposition~\ref{lem:strexp} and Lemma~\ref{lem:Aik} into \eqref{ldder}
we get
\begin{equation}\label{eq:noname2}
 \mu\left(\frac{1}{n}|S_{n}\phi_{i}(x)|> \epsilon\right)
 \lesssim \|\phi_{i, k}\|_{\mathcal H_\alpha}\epsilon^{-1}
 e^{-\tau' n^{\theta'}} + n e^{-\zeta k}
 \end{equation}
where \( \theta'=\theta/(\theta+2) \) and
\(
 \tau'=\min\{\tau, {\epsilon^{2}}/{(162\|\phi_{i,k}\|^2_{\infty})} \}
 \). Notice that taking \( k \) sufficiently large we have in fact
\( \tau'= {\epsilon^{2}}/{(162 k^{2})} \),
and therefore, using the bound on the H\"older norm from Lemma \ref{lem:noname}
and substituting into \eqref{eq:noname2} we have
\[
 \mu\left(\frac{1}{n}|S_{n}\phi_{i}(x)|> \epsilon\right)
 \lesssim
e^{\alpha k} e^{- \epsilon^{2} n^{\theta'}/(162 k^{2})} + n e^{-\zeta k}.
\]
We recall once again that the constant implicit in the inequality \( \lesssim \) is allowed to depend on \( \epsilon \) and on \( \alpha \), even though \( \alpha \) plays no special role in the stretched exponential case.
It is now just a question of making a convenient choice of \( k=k(n) \).
In this case we choose    \( k=n^{\frac{\theta'}{3}-\gamma} \)
and get
\[ e^{\alpha k} e^{- \epsilon^{2} n^{\theta'}/(162 k^{2})} + n e^{-\zeta k}
\leq
e^{(\alpha-\epsilon^{2}n^{3\gamma}/162)n^{\frac{\theta'}{3}-\gamma}}+
n e^{-\zeta n^{\frac{\theta'}{3}-\gamma}}
\]
Now just observe that for any given \( \epsilon \), as long as  \( n \) is sufficiently large we have \( \alpha-\epsilon^{2}n^{3\gamma}/162 < - \zeta \). Since \( \gamma \) can also be chosen arbitrarily small, we obtain the proof of Proposition \ref{prop:DC=>LD-nonempty-CS} in the stretched exponential case.

\end{proof}

\appendix

\section{Special operators and martingales}\label{ap:A}

\subsection{Perron-Frobenius and Koopman operators}\label{subsec:PF-Koop-CE}
Let $(M,\mathcal M, \mu)$ be a probability measure space and $f\colon M\to M$ a measurable map (not necessarily preserving $\mu$).
We say that $f$ is \emph{nonsingular} with respect to $\mu$ if $f_*\nu\ll\mu$ whenever $\nu\ll\mu$.
Given $\varphi\in L^1(\mu)$,  the (signed) measure $\nu_\varphi$ on $\mathcal M$, defined for each
$A\in \mathcal M$ as $$\nu_\varphi(A)=\int_A \varphi d\mu,$$  is clearly absolutely continuous with respect to $\mu$.
 Using the nonsingularity of $f$ we define
 the \emph{Perron-Frobenius operator} $P_\mu:L^1(\mu)\to L^1(\mu)$ by
\begin{equation*}
\label{def:Perron-Frobenius-operator}
 P_\mu\varphi=\frac{df_*\nu_\varphi}{d\mu}.
\end{equation*}
The \emph{Koopman operator}  $U_\mu:L^\infty(\mu)\to L^\infty(\mu)$
is  defined by
\begin{equation*}
\label{def:Koopman-operator}
U_\mu\varphi=\varphi\circ f.
\end{equation*}
Given $\mathcal A$  a sub-$\sigma$-algebra of $\mathcal M$ and $\varphi\in L^1(\mu)$, the (signed) measure $\nu_\varphi^\mathcal A$ on $\mathcal A$, defined for each $A\in\mathcal A$ as
$$
\nu_\varphi^\mathcal A(A)=\int_A \varphi d\mu,
$$  is clearly absolutely continuous with respect to $ \mu|_\mathcal A$.
We finally define the \emph{conditional expectation} $\E_\mu(\cdot|\mathcal A):L^1(\mu)\to L^1(\mu|_{\mathcal A})$ as
\begin{equation*}
 \label{def:conditional-expectation}
 \E_\mu(\varphi|\mathcal A)=\frac{d\nu_\varphi^\mathcal A}{d\mu|_\mathcal A}.
\end{equation*}
Observe that $\E_\mu(\varphi|\mathcal A)$ is the unique $\mathcal A$-measurable function such that for each $A\in\mathcal A$
\begin{equation*}
 \label{eq:cond-exp-property}
 \int_A \E_\mu(\varphi|\mathcal A)d\mu=\int_A \varphi d\mu.
\end{equation*}
Perron-Frobenius and  Koopman  operators enjoy some well-known properties that we collect in (P1)-(P5) below; see e.g. \cite[Chapter~4]{GB}. We observe that in the first two properties we do not need invariance of the measure $\mu$. For all $\varphi\in L^1(\mu)$ we have
\begin{enumerate}

 \item[(P1)]   $\int P_\mu\varphi\, d\mu=\int \varphi d\mu$;

 \item[(P2)]
   $\int (P_\mu\varphi) \psi d\mu=\int \varphi (U_\mu\psi) d\mu$ for all $\psi\in L^\infty(\mu)$.

\end{enumerate}
Moreover, if $\mu$ is $f$-invariant, then for all $\varphi\in L^1(\mu)$ we have
\begin{enumerate}
\setcounter{enumi}{4}
 \item[(P3)]  \label{PF-identity}  $P_\mu  U_\mu\varphi=\varphi$;
   \item[(P4)]  \label{PF-conditional-expectation} 
   $U_\mu^n  P_\mu^n \varphi=\E_\mu(\varphi|f^{-n}(\mathcal M))$ for all $n\ge 1$;
   \item[(P5)]  \label{PF-contraction-all-norms} $\|P_\mu\varphi\|_p\leq \|\varphi\|_p$\,  whenever $\varphi\in L^p(\mu)$ for some $1\leq p\leq\infty$.
\end{enumerate}

\subsection{Filtrations and martingale differences}
Consider a sequence of $\sigma$-algebras $\{\mathcal F_i\}_{i\in\N}$ which forms a \emph{filtration}, meaning that $\mathcal F_i\subset\mathcal F_{i+1}$ for all $i\in\mathbb N$. We say that a sequence of random variables  $\{X_i\}_{i\in\N}$ is \emph{adapted to a filtration} $\{\mathcal F_i\}_{i\in\N}$ if each $X_i$ is measurable with respect to~$\mathcal F_i$.
The following result follows from \cite{Ri} and was drawn in the present formulation from \cite[Proposition~7]{MP}.
\begin{theorem}[Rio]
\label{th:Rio}
 Let $\{X_i\}_{i\in \mathbb N}$ be a sequence of square-integrable random variables adapted to a filtration $\{\mathcal F_i\}_{i\in \mathbb N}$. For all  $1\le p<\infty$ we have
 \[
 \|X_1+\ldots+X_n\|_{2p}^{2 }\leq  4p\sum_{i=1}^n \max_{i\leq u\leq n}\left\|X_i\sum_{k=i}^u \E(X_k|\mathcal F_i)\right\|_p .
 \]
\end{theorem}
We say that random variables  $\{X_i\}_{i\in\N}$ form a sequence of \emph{martingale differences}
with respect to a filtration $\{\mathcal F_i\}_{i\in\N}$ if the sequence is adapted to the filtration and
\begin{equation}
\label{eq:def-martingale-diff}
\E(X_1)=0,\qquad \E(X_{i+1}|\mathcal F_{i})=0,\quad\forall i\ge 1.
\end{equation}
The following result follows from \cite{Az} and  \cite{Ho} and it can be found in the present formulation in~\cite[Theorem~3.1]{LV}.
\begin{theorem}[Azuma-Hoeffding]
\label{th:Azuma-Hoeffding}
 Let $\{X_i\}_{i\in \mathbb N}$ be a sequence of martingale differences.
 If there is $a>0$ such that $\|X_i\|_\infty<a$ for all $1\le i\le n$, then for all $b\in\mathbb R$ we have
 $$\mu\left(\sum_{i=1}^n X_i\geq nb\right)
 \leq\e^{-n\frac{b^2}{2a^2}}.$$
\end{theorem}

\section{Piecewise expanding maps}\label{ap.pecs}

In Theorem \ref{th:exp}  we  consider decay of correlations for observables in a Banach space $\mathcal B$ against observables in $L^1$. In Theorem \ref{th:piecewise} we show that this holds for systems satisfying some general conditions on the Perron-Frobenius operator.
In this appendix we  give more explicit examples of dynamical systems satisfying these conditions. As a consequence, we obtain also exponential large deviations for all these  systems.

\subsection{One-dimensional maps}
The first example is given by  $C^1$ piecewise uniformly expanding maps $f$ on the countable partition ${\mathfrak A}$ of the unit interval $M=[0,1]$, and
verifying the Adler property
$$\sup_{A\in{\mathfrak A}} \sup_{x\in A} \frac{|f''(x)|}{(f'x)^2}<\infty.
$$ In this  case the Lasota-Yorke inequality holds by taking $\mathcal B$
  as
  the space of functions $\varphi$ on the  interval
with bounded total variation $\Huge V_{[0,1]}\varphi$. The corresponding Banach
norm will be given by the sum of  $\Huge V_{[0,1]} \varphi$ plus the
$L^1(m)$ norm of $\varphi$ and this norm is adapted to   $L^1(m)$;
moreover the Banach space just constructed is an algebra.

 Finally, whenever the images under $f$  of the elements in ${\mathfrak A}$ coincide with the whole space $[0,1]$ (Markovian case), the density of the
  acip is bounded from below by a strictly positive constant; see e.g. \cite{AB}. In the general non-Markovian
  situation the  positivity of the density will follow whenever the support of the density will be the whole
  interval (we use here a result by Kowalski \cite{Ko} and Keller \cite{Kel1} which states that if an invariant density $\rho$ is
  lower semicontinuous, then it admits a constant $a>0$ such that $\rho|_{\mbox{supp}\rho}\ge a)$.

 \subsection{Markov maps} The second example generalizes the Markovian case of the previous example. Suppose $\mathfrak{A}$ is a measurable partition of $M$ (not necessarily a Riemannian manifold)  endowed with a probability measure  $m$ on a $\sigma$-algebra  $\mathcal{M}$.
 Let $f: M\to M$ be a measurable map such that $$f(A)\in \sigma(\mathfrak{A}) \quad\text{(mod $m$),} \quad\text{for all $A\in \mathfrak A$,}$$ where $\sigma(\mathfrak A)$ stands for the $\sigma$-algebra generated by $\mathfrak A$.
 We also suppose that $\mathfrak A$ generates $\mathcal{M}$ under $f$ in the sense that
 $\sigma(\bigvee_{n=0}^\infty f^{-n}(\mathfrak A))= \mathcal M$.
 Assume moreover that $f|_A$ is invertible and nonsingular for all $A\in \mathfrak A$. This allows us to define for each $A\in \bigvee_{j=0}^{n-1} f^{-j}(\mathfrak A)$  the inverse branches $g_{A,n}: f^n(A)\rightarrow A$  and the Radon-Nykodym derivaties $\rho_{A,n}={dm\circ g_{A,n}}/{dm}$.   We  assume the following properties:
 \begin{enumerate}
       \item {mixing}: $\forall A,B\in\mathfrak A\,\, \exists n_0\ge0:\, \, f^n(A)\supset B,\,\forall n\ge n_0$;
       \item {big images}: $\inf_{A\in \mathfrak A}m(fA)>0$;
       \item {bounded distortion}: $\exists C>0\,\, \forall n\ge 1\,\,\forall A\in \bigvee_{j=0}^{n-1} f^{-j}(\mathfrak A)\,\, \forall x,y\in f^n(A)$
       $$\left|\frac{\rho_{A,n}(x)}{\rho_{A,n}(y)}-1\right|\le C \theta^{s(x,y)},$$
     \end{enumerate}
 where $\theta$ is some real number in $(0,1)$ and $s(x,y)$ is the separation time defined as in Definition~\ref{def:inducing-scheme}. For these systems we consider the functional space of piecewise Lipschitz functions defined in this way:
 $\varphi:M\rightarrow \mathbb{R}$  is Lipschitz on the set $A\subset M$ if the following seminorm is finite
$$
D_A\varphi\equiv \sup_{x,y\in A}\frac{|\varphi(x)-\varphi(y)|}{\theta^{s(x,y)}}<\infty.
$$
Letting $\mathfrak{B}$ be the partition such that $\sigma(f(\mathfrak A))=\sigma(\mathfrak B)$, we define $D_{\mathfrak B}\varphi=\sup_{A\in \mathfrak B}D_A\varphi$.  Finally we define $$\mathcal L=\{\varphi\in L^{\infty}(m): D_{\mathfrak B}\varphi<\infty\}, $$ equipped with the norm
$$
\|\varphi\|_{\mathcal L}:= \|\varphi\|_{L^{\infty}(m)}+D_{\mathfrak B}\varphi.
$$
This norm is adapted to $L^1(m)$.

On the space $\mathcal L$ the Perron-Frobenius operator satisfies the Lasota-Yorke inequality and the density of the invariant measure will be $m$-almost everywhere bounded away from zero; see \cite{AD}.

\subsection{Multidimensional maps}  The third interesting example is given by multidimensional piecewise uniformly expanding maps for which we will use the space of quasi-H\"older functions   described below. Since we are going to prove for such maps a few apparently new results, we need to define them carefully; we would like to stress first that Markov maps are a special case of them. We follow here the definition proposed by Saussol \cite{Sa}; these maps have also been investigated by Blank \cite{Bl}, Buzzi \cite{Bu}, Buzzi and Keller \cite{BK} and Tsuji \cite{Ts}; the situation where the expansion is not anymore uniform has been investigated in the paper \cite{HV}.

Let $M\subset \Bbb R^N$ be a compact subset with
$\overline{\interior M}=M$ and $f: M\to M$.
 For $A\subset M$ and $\e>0$ we put $B_{\e}(A)=\{x\in \Bbb R^N:
d(x, A)\le \e \}$, where  $d$ be the Euclidean distance in $\Bbb R^N$. Assume that there exist at most countably many   disjoint open sets $U_i$  such that
 $m(M\setminus\bigcup_{i=1} {U_i})=0$, where $m$ denotes Lebesgue measure in the Borel sets of $\Bbb R^N$. Assume moreover that there are open sets  $\widetilde U_i\supset \overline{U_i}$ and $C^{1+\alpha}$ maps $f_i: \widetilde U_i\to\Bbb R^N$ such that $f_i|_{U_i}=f|_{U_i}$ for each $i$. Suppose that there are constants $c,\varepsilon_1>0$ and $0<\alpha<1$ such that the following hold:
\begin{enumerate}
\item
$f_i(\widetilde U_i)\supset B_{\e_1}(f(U_i))$ for each $i$;
\item
for each $i$ and $x, y\in f(U_i)$ with
$d(x,y)\le \e_1$,
$$
\bigl|\det Df_i^{-1}(x)-\det Df_i^{-1}(y)\bigr| \le c |\det
Df_i^{-1}(x)|d(x,y)^\a;
$$
\item there exists $s=s(f)<1$ such that
$$
\sup_i\sup_{x\in f_i(\widetilde{U}_i)} \left\|Df_i^{-1}(x)\right\|<s;
$$
\item each $\partial U_i$ is a codimension one embedded compact $C^1$  submanifold and
\begin{equation}\label{ineq:esq}
s^{\alpha}+\cfrac{4s}{1-s}Y(f)\cfrac{\gamma_{N-1}}{\gamma_{N}}<1,
\end{equation}
where $Y(f)=\sup_{x}\sum_i \# \left\{\text{smooth pieces
intersecting $\partial U_i$ containing $x$}\right\}$
and $\gamma_N$ is the
volume of the unit ball in $\mathbb{R}^N$. 
\end{enumerate}
According to \cite{Sa}, condition \eqref{ineq:esq} can be weakened. We nevertheless keep that condition which is particularly simple to handle with when the boundaries of the $U_i$ are smooth.
Given a Borel set $\Omega\subset M$, we define the oscillation of $\varphi\in L^1(m)$ over
$\Omega$ as
$$
\osc(\varphi,\Omega) := \Esup\varphi|_\Omega - \Einf\varphi|_\Omega.
$$
Letting $B_{\epsilon}(x)$ denote the ball of radius $\epsilon$ around
the point $x$, we get a measurable function $x\rightarrow
\mbox{osc}(\varphi, \ B_{\epsilon}(x))$.
Given $0<\a <1$ and $\e_0>0$, we define the $\alpha$-seminorm of $\varphi$ as
\begin{eqnarray}\label{seminorm}
|\varphi|_{\alpha} = \sup_{0<\epsilon\le\epsilon_0}
\epsilon^{-\alpha} \int_{\Bbb R^N} \mbox{osc}(\varphi,
B_{\epsilon}(x)) dm(x).
\end{eqnarray}
We consider the space of the functions
with bounded $\alpha$-seminorm
\begin{eqnarray}\label{norm}
V_{\alpha} =  \left\{ \varphi\in L^1(m): |\varphi|_{\alpha} <\infty \right\},
\end{eqnarray}
and equip $V_{\alpha}$ with the norm
\begin{eqnarray}\label{falphanorm}
\|\cdot\|_{\alpha} =
\parallel\cdot\parallel_{L^1(m)} + |\cdot|_{\alpha}.
\end{eqnarray}
 We remark that this space does not
depend on the choice of $\epsilon_0$ and
 $V_{\a}$ is a Banach space endowed with the norm $\parallel\cdot\parallel_{\alpha}$.
Moreover, according to Theorem~1.13 in \cite{Kel2}, the unit ball in
$V_\alpha$ is compact in $L^1(\mu)$.

The assumptions (1)-(4) above allow us to get a Lasota-Yorke inequality when the Perron-Frobenius operator is applied to
functions belonging to the space $V_\alpha$; see \cite{Bl} and \cite{Kel2}
for the introduction of such a space in the theory of dynamical systems.

\subsection{Decay of correlations}
\label{sec:piecewise} Here we prove Theorem~\ref{th:piecewise}.
It is well known that under conditions (1)-(4) in Section~\ref{se.perron}, the Ionescu-Tulcea-Marinescu theorem \cite{ITM} asserts that the operator $\p$ is quasi-compact and this implies  the existence of  an invariant probability measure
$\mu$ for the map $f$ which is absolutely continuous with respect to $m$ on $M$ and with density $h\in \mathcal B$.
The measure $\mu$ has a finite number of ergodic components, and it
is the ``unique greatest'' in the sense that any other
measure absolutely continuous with respect to $m$ is absolutely
continuous with respect to $\mu$. Moreover, $M$ is partitioned $\mu$ mod 0 into a finite number of measurable sets upon which a certain power of $f$ is mixing. Since we are mostly interested in the rate of decay of correlations, we will suppose that $M$ is the only mixing component for $f$. The iterates of the Perron-Frobenius operator enjoy  the following spectral decomposition:
 \begin{equation}\label{eq:piq}
   \p^n= \Pi + Q^n,
 \end{equation}
 where $\Pi$ projects
$\varphi\in \mathcal B$ into the fixed points of $\p$,
\begin{equation}\label{eq:pih}
 \Pi(\varphi)=h\int \varphi dm,
\end{equation}
and the linear operator $Q$ verifies
\begin{equation}\label{ineq:qn}
\|Q^n(\varphi)\|_B\le C''q^n \|\varphi\|_B,
\end{equation}
  where $C''>0$ and $0<q<1$ are constants depending on $f$.

Now, take $\varphi\in\mathcal B$ and
assume with no loss of generality that $\int\varphi d\mu=0$, or equivalently $\int\varphi hdm=0$, where $h=d\mu/dm$.
Since $h\in\mathcal B$, by property (5) above we have that $\varphi h\in \mathcal B$. Therefore, using \eqref{eq:piq}, \eqref{eq:pih}, \eqref{ineq:qn} and property (6) above, for any $\psi\in L^1(m)$ we have
\begin{align*}
  \left|\int \varphi \psi\circ f^n d\mu\right|&=\left|\int \psi \p^n(\varphi h) dm\right|\\
  &\le\left|\int\psi Q^n(\varphi h)dm\right|\\
  &\le C'\|\psi\|_{L^1(m)}  \|Q^n(\varphi h)\|_{\mathcal B}\\
  &\le C' C''q^n ||\psi||_{L^1(m)}||h\varphi||_{\mathcal B}.
\end{align*}
Recalling that by assumption there is $c>0$ such that $h\ge c$, it then follows that
$$\|\psi\|_{L^1(m)}=\int\frac{|\psi|}h  d\mu\le \frac{1}{c}\|\psi\|_{L^1(\mu)}.$$
Thus we have
$$\cv_{\mu}(\varphi, \psi\circ f^{n})\le  \frac{1}{c}C' C'' \|h\|_{\mathcal B}\,q^n,$$
which is obviously summable in $n$.

\subsection*{Acknowledgments}
We wish to thank Ian Melbourne for very valuable observations and
relevant suggestions. We are grateful to  Neil Dobbs, Carlangelo Liverani, Matt Nicol and  Mike Todd
for fruitful conversations and comments.


\begin{thebibliography}{MPU}
\providecommand{\bibinfo}[2]{#2}

\bibitem [AD]{AD} J Aaronson, M Denker, \emph{Local limit theorems for partial sums of stationary sequences generated by Gibbs-Markov maps}, Stochastic Dyn., {\bf 1}, (2001), 193-237

\bibitem[Al1]{Alv00}    J.F. Alves, \emph{SRB measures for non-hyperbolic systems with multidimensional expansion},
              Ann. Scient. Éc. Norm. Sup. (4) \textbf{33} (2000), no. 1, 1-32.

\bibitem[Al2]{Alv} J.F. Alves,
{\em Strong statistical stability of non-uniformly expanding maps,}
Nonlinearity {\bf 17} (2004) 1193-1215.

 \bibitem[AA]{AA}      J. F. Alves,    V. Araújo,    \emph{Random perturbations of nonuniformly expanding maps}.
 Astérisque \textbf{286} (2003), 25-62.

\bibitem[ABV]{ABV} J.F. Alves, C. Bonatti, M. Viana,
{\em SRB measures for partially hyperbolic systems whose central
direction is mostly expanding,}  Invent. Math.  {\bf 140} (2000)
351-398.




\bibitem[ADL]{AlvDiaLuz}    J.F. Alves,    C.L. Dias,     S. Luzzatto,     \emph{Geometry of expanding absolutely continuous invariant measures and the liftability problem}, preprint 2009.

\bibitem[ALP]{ALP} J.F. Alves, S. Luzzatto, V. Pinheiro,
{\em    Markov structures and decay of correlations for
non-uniformly expanding dynamical systems,}  Ann. Inst. H.
Poincar\'{e} Anal. Non Lin\'eaire  {\bf 22} (2005) 817-839.

\bibitem[AV]{AV}     J. F. Alves,   M. Viana,     \emph{Statistical stability for robust classes of maps with non-uniform expansion}.
Ergodic Theory \& Dynam. Systems \textbf{22 }(2002), no. 1, 1-32.

\bibitem[AP]{AraPac} V. Ara\'ujo, M.J. Pacifico, \emph{Large deviations for non-uniformly expanding maps}. J. Stat. Phys. \textbf{125 }(2006) 415--457.

\bibitem[Az]{Az} K. Azuma,  \emph{Weighted sums of certain dependent random variables}. T\^ohoku Math. J. \textbf{19 }(1967) 357--367.

\bibitem[Bl]{Bl}M. Blank,
\emph{Stochastic properties of deterministic dynamical systems},
Sov. Sci. Rev. C Maths/Phys., {\bf 6}, (1987), 243-271

\bibitem[Br]{AB} A. Broise, \emph{Transformations dilatantes de l'intervalle et th\'eor\`emes limites}, Asterisque, {\bf 238}, (1996), 5-110.

\bibitem[BLS]{BLS} H. Bruin, S. Luzzatto. S. van Strien,
{\em Decay of correlations in one-dimensional dynamics,}
 Ann. Sci. Ecole Norm. Sup. {\bf 36} (2003) 621-646.

\bibitem [Bu]{Bu}J. Buzzi, \emph{Absolutely continuous  invariant probability measures for arbitrary expanding piecewise $\bf R$-analytic mappings of the plane},  Ergodic Theory Dynam. Systems, {\bf 20}, (2000), 697-708


\bibitem [BK]{BK}J. Buzzi, G. Keller, \emph{Zeta functions and transfer operators for multidimensional piecewise affine and expanding maps}, Ergodic Theory Dynam. Systems, {\bf 21}, (2001), 689-716

\bibitem[BST]{BST} J. Buzzi, O. Sester and M. Tsujii, \emph{Weakly expanding skew-products of quadratic maps}. Ergod. Th. \& Dynam. Sys. \textbf{23}  (2003), 1401--1414.


\bibitem[DHL]{DHL}
K. D\'iaz-Ordaz, M. Holland, S. Luzzatto,
\emph{Statistical properties of one-dimensional maps with critical points and singularities},
 Stoch. Dyn. {\bf 6} (2006), 423-458.

\bibitem[GB]{GB} P. G\'ora, A. Boyarsky,  \emph{Laws of Chaos. Invariant measures and Dynamical Systems in one dimension.}, Birkh\"auser, 1997.

\bibitem[Go]{Gou} S.\ Gou\"ezel,
{\em Decay of correlations for nonuniformly expanding systems,}
Bull. Soc. Math. France {\bf 134} (2006) 1--31.


\bibitem[Hoe]{Ho} W. Hoeffding, \emph{Probability inequalities for sums of bounded random variables}. J. Amer. Statist. Assoc. \textbf{58} (1963), 13--30.

\bibitem[Hol]{Hol}M. Holland,
\emph{Slowly mixing systems and intermittency maps},
 Ergodic Theory Dynam. Systems {\bf 25} (2005),  133-159.



\bibitem[Hu]{Hu} H. Hu, \emph{Decay of correlations for piecewise smooth maps with indifferent fixed points}.  Ergodic Theory Dynam. Systems  \textbf{24}  (2004),  no. 2, 495--524.

\bibitem [HV]{HV} H. Hu, S. Vaienti, \emph{Absolutely continuos invariant measures for non-uniformly expanding maps}, Ergodic Theory Dynam. Systems , {\bf 29}, (2009), 1185-1215


\bibitem[IM]{ITM} C. Ionescu-Tulcea, G. Marinescu, \emph{Th\'eorie ergodique pour des classes d'op\'eratations non compl\'etement continues}, Ann. Math., {\bf 52}, (1950), 140-147

\bibitem[Ke1]{Kel1} G. Keller, \emph{Piecewise monotonic transformations and exactness}, Collection : Seminar on Probability, Rennes, 1978 (French) ; Exp. {\bf 6}, 32 pp. Univ. Rennes, (1978)

\bibitem [Ke2]{Kel2} G. Keller, \emph{Generalized bounded variation and applications to piecewise monotonic transformations}, Z. Wahr. verw. Geb., {\bf 69}, (1985), 461-478

\bibitem [Ko]{Ko} Z.S. Kowalski, \emph{Invariant measure for piecewise monotonic trasformations has a lower bound on its support}, Bull Acad. Pol. Sci. Math., {\bf 27}, (1979), 53-57

\bibitem[LV]{LV} E. Lesigne, D. Voln\'y. \emph{Large deviations for martingales}. Stoch. Proc. Applns. \textbf{96} (2001) 143--159.

\bibitem[LSV]{LivSauVai}   C. Liverani, B. Saussol, S. Vaienti, \emph{A probabilistic approach to intermittency}.  Ergodic Theory Dynam. Systems  \textbf{19}  (1999),  no. 3, 671--685.

\bibitem[Me]{M} I.\ Melbourne. \emph{Large and moderate deviations for slowly mixing dynamical systems}, Proc. Amer. Math. Soc. 137 (2009) 1735-1741

\bibitem[MN]{MN} I.\ Melbourne, M.\ Nicol.
\emph{Large deviations for nonuniformly hyperbolic systems,} Trans. Amer. Math. Soc. 360 (2008) 6661-6676.

\bibitem[MPU]{MP} F. Merlev\'ede, M. Peligrad, S. Utev. \emph{Recent advances in invariance principles for stationary
sequences}. Probab. Surv. \textbf{3} (2006) 1--36 (electronic).

\bibitem[PM]{PomMan} Y. Pomeau, P. Manneville, \emph{Intermittent transition to turbulence in dissipative dynamical systems}.  Comm. Math. Phys.  \textbf{74}  (1980), no. 2, 189--197.

\bibitem[RY]{RY}
    L. Rey-Bellet, L.S. Young, \emph {Large deviations in non-uniformly hyperbolic dynamical systems}, Ergodic Theory Dynam. Systems {\bf 28} (2008) 587--612.


\bibitem[Ri]{Ri} E. Rio, \emph{Th\'eorie asymptotique des processus al\'eatoires faiblement d\'ependants}. Math\'ematiques \& Applications (Berlin) [Mathematics \& Applications] 31, Springer-Verlag, Berlin, 2000.

\bibitem[Sar]{Sar} O. Sarig, \emph{Subexponential decay of correlations}.  Invent. Math.  \textbf{150 } (2002),  no. 3, 629--653.

\bibitem[Sau]{Sa} B. Saussol, \emph{Absolutely continuous invariant measures for multidimensional expanding maps}, Israel J. Math., {\bf 116}, (2000), 223-248.

\bibitem [Ts]{Ts} M. Tsujii,
\emph{Absolutely continuous invariant measures
for expanding piecewise linear maps},
Invent. Math., {\bf 143}, (2001), 349-373

\bibitem[Vi]{V}  M.  Viana, \emph{Multidimensional nonhyperbolic attractors}.  Inst. Hautes Études Sci. Publ. Math.  No. \textbf{85 } (1997), 63--96.

\bibitem[Yo1]{Y2}  L.S.\ Young,
{\em Statistical properties of dynamical systems with some
hyperbolicity,} Ann. of Math. (2) {\bf 147}  (1998) 585--650.

\bibitem[Yo2]{Y3}L.S.\ Young,
{\em  Recurrence times and rates of mixing,} Israel J. Math. {\bf
110}  (1999) 153--188.


\end{thebibliography}
\end{document}